\documentclass[12 pt, leqno]{amsart}
\usepackage{mathtools,amssymb}
\usepackage{enumerate}
\usepackage{float}
\usepackage[utf8]{inputenc}
\usepackage[english]{babel}
\restylefloat{table}
\DeclarePairedDelimiter{\ceil}{\lceil}{\rceil}
\newcommand{\be}{\begin{equation}}
\newcommand{\ee}{\end{equation}}

\title{Thue's inequalities and the hypergeometric method}

\author{Shabnam Akhtari}
\address{Department of Mathematics\\
Fenton Hall\\
University of Oregon\\
Eugene, OR 97403-1222 USA}
\email {akhtari@uoregon.edu}

\author{N. Saradha}
\address{School of Mathematics\\
Tata Institute of Fundamental Research\\
Mumbai (India) 400 005}
\email{saradha@math.tifr.res.in}

\author{Divyum Sharma}
\address{School of Mathematics\\
Tata Institute of Fundamental Research\\
Mumbai (India) 400 005}
\email{divyum@math.tifr.res.in}

\date{\today}                               
                                                                                                 


\begin{document}

\newtheorem{thm}{Theorem}[section]
\newtheorem{prop}[thm]{Proposition}
\newtheorem{lemma}[thm]{Lemma}
\newtheorem{cor}[thm]{Corollary}
\newtheorem{conj}[thm]{Conjecture}
\newtheorem{defin}{Definition}[section]

\begin{abstract}
We establish upper bounds for the number of primitive integer
solutions to  inequalities of the shape $0<|F(x, y)| \leq h$, where $F(x , y) =(\alpha x + \beta y)^r -(\gamma x + \delta y)^r \in \mathbb{Z}[x ,y]$,  $\alpha$, $\beta$, $\gamma$ and $\delta$  are algebraic constants with
$\alpha\delta-\beta\gamma \neq 0$, and  $r \geq 5$ and 
$h$ are integers. 
As an important application, we pay special attention to  binomial Thue's inequaities  $|ax^r - by^r| \leq c$. The proofs are based on the hypergeometric method of Thue and Siegel and its refinement by Evertse. 
\end{abstract}

\maketitle

\section{Introduction and statements of the results}\label{Intro}

 As a consequence of his improvement of Liouville's theorem on approximation of algebraic numbers by rationals, Thue \cite{Thu1} proved that if $F(x,y)$ is a binary form  with integer coefficients, having at least three pairwise non-proportional linear factors in its factorization over $\mathbb{C}$, and $h$ is
a non-zero integer then the Diophantine equation 
$$F(x,y)=h$$
has only a finite number of integer solutions.  Such equations are called Thue equations, and inequalities of the shape
\begin{equation}\label{firstineq}  
0<|F(x,y)| \leq h
\end{equation}
are called Thue inequalities.

Thue used Pad\'{e} approximation to binomial functions to study some families of Thue equations (see for example \cite{Thu2}). 
Later Siegel \cite{Sie2} identified the
approximating polynomials in  Thue's method  as
hypergeometric polynomials. 
The hypergeometric  method  of Thue and Siegel applies to a special family of Thue equations and inequalities ( see, for example, \cite{ Akh0}, \cite{Sie1} and \cite{Vou}).

\begin{defin} \label{defofdiag} A binary form $F(x , y) \in \mathbb{Z}[x , y]$ is called diagonalizable if  it can be written as
\be\label{form}
	 F(x , y) = (\alpha x + \beta y)^r - (\gamma x + \delta y)^r, 
 \ee 
 where the constants  $\alpha$, $\beta$, $\gamma$ and $\delta$  satisfy
\be\label{integralcondition}
 j=\alpha \delta -  \beta \gamma\neq 0.
\ee
\end{defin}

If  $F(x , y) = (\alpha x + \beta y)^r - (\gamma x + \delta y)^r \in \mathbb{Z}[x , y]$ is diagonalizable  then 
	\begin{equation}\label{chiquad}
(\alpha x + \beta y)  (\gamma x + \delta y) = \chi (Ax^2 + B xy+ Cy^2)
\end{equation}
for some $A, B, C \in \mathbb{Z}$ and a constant $\chi$. 
Let 
	\begin{equation}\label{defofD}
	D = D(F) = B^2 - 4 A C.
	\end{equation}
Then
	$$
	j^2 = \chi^2 D.
	$$
Therefore  $D(F) \neq 0$ for a diagonalizable form $F(x, y)$. We will denote the discriminant of $F(x, y)$ by  $\Delta(F)$.


The most interesting family of diagonalizable forms are binomial forms, the forms of the shape $ax^r-by^r$. Before we state our new general theorems on diagonalizable forms, we will present the applications of our theorems to obtaining bounds for the number of solutions to Thue's inequality with binomial forms. 
Many mathematicians, including Thue \cite{Thu2}, Siegel \cite{Sie2}, Domar \cite{Dom},
Evertse \cite{Eveabc}, Bennett and de Weger \cite{BeWe}, studied  the equation
$ax^r-by^r= c$.
In a breakthrough work \cite{Ben1},  Bennett  used a sophisticated combination of the  hypergeometric method with Chebyshev-like estimates for primes in arithmetic progressions to show 
 that
the  equation $$a x^r - by^r = 1,$$
 with $a$ and $b$ positive,   has at most one solution in positive
integers $x,y$. This is a sharp result, as the
equation 
$$(a+1)x^r-ay^r=1$$
has precisely one solution $(1,1)$ in positive integers, for every positive integer $a$. In \cite{Sie2}, Siegel showed that the equation $0< |a x^r - by^r| \leq c$ has at most one primitive solution in positive integers $x$ and $y$ if
$$
\left|a b\right|^{\frac{r}{2} - 1} \geq 4 \left(r \prod_{p| r} p^{\frac{1}{p-1}}\right)^{r} c^{2r -2}.
$$
Our main results can be directly applied to binomial Thue's inequalities to  improve the above result of Siegel, by extending the range of $c$ with respect to $a$ and $b$.  For example, Corollary \ref{cork=3} implies the following. 
\begin{thm}\label{Binoineq}
Let $a$, $b$ and $c$ and $r \geq 5$ be positive integers.  Assume that 
$$
ab \geq 2^r r^{\frac{7r}{r-4}} c^{2+\frac{r^2 + r+2}{r(r-4)}}.
$$
Then the inequality 
$$0< |a x^r - by^r| \leq c
$$
has at most $3$ primitive solutions in positive integers $x$ and $y$.  
\end{thm}

Furthermore, an important  observation in Section \ref{DTE}, will allow us to use a reduction method to obtain upper bounds for the number of integer solutions of equations of the shape $ax^r - by^r = c$. This reduction method has been used by Bombieri  and  Schmidt in \cite{Bos} and Stewart in \cite{Ste} for general Thue equations. A simple way to describe the application of this elaborated method is to assume that  $K$ is an upper bound for the number of solutions to equations $G(x , y) =1$, where $G$ ranges over all irreducible integral binary forms of degree $r \geq 3$. Then the number of primitive solutions of $F(x , y) = h$, for every irreducible $F(x , y) \in \mathbb{Z}[x , y]$ of degree $r$, is at most $K r^{\omega(h)}$, where $\omega(h)$ denotes the number of prime factors of $h$. Now suppose that we are interested in the number of primitive solutions of $ax^r - by^r = c$. Having the beautiful result of Bennett, mentioned above, on the number of solutions of $ax^r - by^r =1$, one would hope to get a good bound in the case $ax^r - by^r = c$. However, this is not straightforward, because the reduction method that was used for general Thue equations, will receive $ax^r - by^r = c$ and reduce it to a number of Thue equations $\tilde{F}(x, y) = 1$, where the forms $\tilde{F}$ are of degree $r$ but are not necessarily  binomials! It turns out that once we start with a diagonalizable form, the reduction method will provide new diagonalizable forms. Therefore thinking of $ax^r - by^r = c$ as a diagonalizable form has a great advantage. We will show that
\begin{thm}\label{binomial=}
Let $a$, $b$ and $c$ and $r \geq 5$ be positive integers,  with $\gcd(c, r a b) =1$.  Assume that 
$$
ab \geq 2^r r^{\frac{182.6 r}{r-1}}.
$$
Then the equation
$$ a x^r - by^r = c
$$
has at most $2\,  r^{\omega(c)}$ primitive solutions in positive integers $x$ and $y$.  Further, if 
$$
 ab \geq 2^{r} r^{7r/(r-4)}
 $$
 then the equation
$$ a x^r - by^r = c
$$
has at most $3\,  r^{\omega(c)}$ primitive solutions in positive integers $x$ and $y$.
\end{thm}

In \cite{Eveabc}  Evertse showed that, for positive integers $a$, $b$ and $c$ the equation
$ a x^r - by^r = c$ has at most $2 r^{\omega(c)} + 6$ solutions. 
In \cite{Mig}, using the theory of linear forms in logarithms,  Mignotte obtained effective results for the size of the solutions of binomial Thue's inequality.  Some computational results on binomial Thue equations can be found in \cite{BBGP}.


\begin{defin}\label{defofprimitive}
Let $(x , y) \in \mathbb{Z}^2$  satisfy the equation $F(x , y) = h$ (or the inequality $0<|F(x,y)| \leq h$). We call $(x , y)$ a primitive solution if  $\gcd (x , y) =1$.
\end{defin}
Throughout this manuscript, we deem 
$(x,y)$ and $(-x,-y)$ as one solution.
\begin{defin}
For an integer $h$
and a binary form $F(x ,y) \in \mathbb{Z}[x,y]$,
let $N_F(h)$ denote the number of primitive solutions to the
inequality $0< |F(x , y)| \leq h$. For $Y_L\in\mathbb{R}^+$, let $N_F(h;Y_L)$ denote the number of primitive solutions  of $0< |F(x , y)| \leq h$ with $y\geq Y_L$.
 Put
	\begin{equation}\label{Deltaprime}
	\Delta'=\frac{|\Delta|}{2^{r^2-r} r^r h^{2r-2}},	
	\end{equation}
	where $\Delta$ is the discriminant of $F(x , y)$.
	\end{defin}

We will follow some ideas of  Siegel in \cite{Sie3} to prove several results on diagonalizable Thue's inequality. In Section \ref{compareSiegel} we will state the main theorem in \cite{Sie3} and compare it with one of our main theorems.

	 \begin{thm}\label{2r1}
	 Suppose $F(x , y)$ is a diagonalizable form with degree $r \geq 6$ and discriminant $\Delta$.
	 Assume that
	 	\begin{equation}\label{delta2r1}
	 	\Delta' \geq r^{13r^2(r-1)/(r^2-5r-2)} 
h^{4(r-1)(r^2-r+2)/(r^2-5r-2)},
\end{equation}
where $\Delta'$ is defined in \eqref{Deltaprime}.
	Then 
	\[
	N_F(h)\leq \begin{cases} 2r+1 &\mbox{if } D<0 \\ 
5 & \mbox{if } D>0, \ r \textrm{ is even and $F$ is indefinite}\\
3 & \mbox{if } D>0, \ r \textrm{ is odd and $F$ is indefinite}\\
1 & \mbox{if } D>0 \textrm{ and $F$ is definite}.
\end{cases} 
	\]
	 \end{thm} 

	\begin{thm}\label{thml}
 Suppose $F$ is a diagonalizable form with degree $r\geq 5$ and discriminant $\Delta$.
Assume that 
	\begin{equation}\label{Delta_condn}
	\Delta' \geq  r^{\alpha_1}h^{\alpha_2}
	\end{equation}
with
	\begin{equation}\label{twoalphas}
	\alpha_1=\frac{7r^2(r-1)}{(r-1)^{m-1}-2r-1}\ \textrm{ and }\
	\alpha_2=\frac{(r-1)(r^2+r+2)}{(r-1)^{m-1}-2r-1},
	\end{equation}
	where $m \geq 3$ is an integer and $\Delta'$ is defined in \eqref{Deltaprime}.
Then
	\be\label{kr}
	N_F(h)\leq \begin{cases} rm &\mbox{if } D<0 \\ 
2m & \mbox{if } D>0, \ r \textrm{ is even and $F$ is indefinite}\\
m & \mbox{if } D>0, \ r \textrm{ is odd and $F$ is indefinite}\\
1 & \mbox{if } D>0 \textrm{ and $F$ is definite}.
\end{cases} 
	\ee
	\end{thm}
	
	One interesting feature of Theorem \ref{thml} is that it shows explicitly how by increasing the range  for the positive integer $h$ (by increasing the value of $m$ and therefore decreasing the value of $\alpha_{2}$), we are required to count more solutions. The upper bound on the  number of solutions increases linearly in terms of $m$ while the range of $h$ increases double exponentially in $m-4$, provided that $h \ll |\Delta|^{\frac{1}{2(r-1)}}$ and $m\geq 4$.

By taking $m =3$, we obtain the following immediate corollary of Theorem \ref{thml}.
	\begin{cor}\label{cork=3}
	Suppose $F$ is a diagonalizable form with degree $r \geq 5$.
	  If
	\[
		\Delta' \geq  r^{7r(r-1)/(r-4)} h^{(r-1)(r^2+r+2)/(r(r-4))} ,
	\]
	then
	\[
	N_F(h)\leq \begin{cases} 3r &\mbox{if } D<0 \\ 
6 & \mbox{if } D>0, \ r \textrm{ is even and $F$ is indefinite}\\
3 & \mbox{if } D>0, \ r \textrm{ is odd and $F$ is indefinite}\\
1 & \mbox{if } D>0 \textrm{ and $F$ is definite}.
\end{cases} 
	\]
	\end{cor}

Our assumption $r \geq 5$ is to simplify the proofs and reduce the amount of computations. One can use similar techniques for the diagonalizable forms of  degree $r= 3, 4$. However, these special and important  cases have been studied before. We refer the reader to \cite{AkhC1, SS2, Wak1, Wak2} for cubic inequalities and to \cite{Akhint}, in particular its Theorem 1.4, for quartic inequalities. Also Wakabayashi has studied a family of parametric quartic Thue's inequalities in \cite{Wakq1, Wakq2}.

There have been a few  results on the number of solutions to general Thue's inequalities $0<|F(x,y)| \leq h$   which assert that  if  $h$ is small in comparison with a function of  $|\Delta(F)|$ then   upper bounds  $N_F$, independent of $h$,   could be established (see for example \cite{Akhsm, SS1, Gy1}). In particular, the results in \cite{Akhsm} and \cite{Gy1} imply that such upper bounds can be obtained   if $h < |\Delta|^{\frac{1}{4(r-1)}}$. 
Our method allows us to improve these upper bounds for diagonalizable Thue's inequalities. Moreover, we are able to improve the dependency of $h$ on the discriminant of the binary form in the inequality.

The following is another corollary of Theorem \ref{thml}.
\begin{cor} \label{newcor}
Let $F(x , y)$ be a diagonalizable form of degree  $r \geq 5$, and  $\epsilon$ any positive number  with
$$
\frac{(r^2+r+2)}{4(r-1)[(r-1)^{m-1}-2r-1]}<\epsilon<\frac{1}{2(r-1)}.
$$
If
\begin{equation}\label{hsatisabove}
0<h\leq \frac{|\Delta|^{\frac{1}{2(r-1)}-\epsilon}}{2^{r/2}r^{7}},
\end{equation}
then $N_F(h)$ satisfies \eqref{kr}. In particular, if $h$ satisfies
the  inequality \eqref{hsatisabove}, $D<0$ 
and $0<\epsilon<\frac{1}{2(r-1)},$ 
then
$$N_F(h) \leq \left(4+\ceil[\Bigg]{\frac{\log{\left(1/\epsilon\right)} - \log 4}{\log(r-1)} }\right)r $$
where the symbol $\ceil{x}$ denotes the smallest integer 
greater than or equal to the real number $x$.
\end{cor}

A nice feature of diagonalizable forms is that they can adopt a reduction theory based on the classical reduction for quadratic forms. We will define reduced diagonalizable forms in Section \ref{EF}. We will prove the following theorems on the number of large solutions to Thue inequalities $0<|F(x,y)| \leq h$, without any assumption on the size of $h$.

\begin{thm}\label{thmy}
Suppose $F(x , y)$ is a reduced diagonalizable form  of degree $r \geq 6$ and with $D<0$. Let $m \geq 2$
be an integer. Let
	\begin{equation}\label{ycondn}
	Y_L= \frac{r^{i_1} h^{i_2}}{|j|^{i_3}}
        \end{equation}
where
$$i_1=2+\frac{2}{r},\
i_2= \frac{1}{r-2}+\frac{r-3}{(r-2)(r-1)^{m-1}}$$
and
\[
i_3=\begin{cases}
0 &\mbox{if } |j| \geq 1\\
\frac{r}{2(r-2)} & otherwise.
\end{cases}
\]
Then $$N_F(h;Y_L)\leq m r.$$
\end{thm}

\begin{thm}\label{thmH}
Suppose $F(x , y)$ is a diagonalizable form  of degree $r \geq 5.$ Let $m\geq 3$ be an integer and
\be\label{HL}
	H_L=r^{i_4}h^{i_5}|j|^{i_6},
\ee
with
	\[
	i_4=5+\frac{11r-3}{(r-1)^{m-1}},\ i_5=2+\frac{2(r-3)}{(r-1)^{m-1}}
	\]
and
	\[
	i_6=\begin{cases}
		2 &\mbox{if } |j|\geq 1\\
		\frac{2}{(r-3)(r-1)^{m-1}} &\mbox {if } |j|<1.
		\end{cases}
	\]
	Let $
	 H(x,y) = F_{xx} F_{yy} - F^{2}_{xy}
	 $ be the  Hessian  of $F(x , y)$.	 
Then the number of solutions $(x,y)$
of \eqref{firstineq} satisfying $|H(x,y)|\geq H_L$ is bounded by
$$
\begin{cases} rm &\mbox{if } D<0 \\ 
2m & \mbox{if } D>0, \ r \textrm{ is even and $F$ is indefinite}\\
m & \mbox{if } D>0, \ r \textrm{ is odd and $F$ is indefinite}\\
1 & \mbox{if } D>0 \textrm{ and $F$ is definite}.
\end{cases} 
$$

\end{thm}

Finally, we will prove a theorem on diagonalizable Thue equations.

\begin{thm}\label{BSchfordiag}
Suppose $F$ is a diagonalizable form with degree $r \geq 5$ and discriminant $\Delta$. Let $h$ be an integer such that
 $\gcd(h , \Delta) = 1$.
	  If
	\[
		|\Delta| \geq 2^{r^2-r} r^{r+7r(r-1)/(r-4)},
	\]
where $\Delta$ is the discriminant of $F$, then the number of solutions of the equation $|F(x , y)| =h$
 is bounded by
	\[
	 \begin{cases} 3r^{1+\omega(h)} &\mbox{if } D<0 \\ 
6 r^{\omega(h)}& \mbox{if } D>0, \ r \textrm{ is even and $F$ is indefinite}\\
3  r^{\omega(h)}& \mbox{if } D>0, \ r \textrm{ is odd and $F$ is indefinite}\\
 r^{\omega(h)} & \mbox{if } D>0 \textrm{ and $F$ is definite},
\end{cases} 
	\]
	where $\omega(h)$ denotes the number of prime divisors of $h$.

\end{thm}

The outline of this manuscript is as follows. 
In Section \ref{compareSiegel} we compare Siegel's main theorem in \cite{Sie3} to our Theorem  \ref{thml}. In Section \ref{pre} we recall some known facts that  are going to be used in our proofs. In Section \ref{EF} we introduce reduced diagonalizable forms and some of their properties. In Section \ref{GP} we will establish some important gap principles which will be a major part of our proofs. 
Another essential ingredient in establishing our results is the 
 use of the  hypergeometric
method together with the construction of some  sequences of
algebraic numbers. Our proofs are based on the work of  Siegel in \cite{Sie3} and its improvements for cubic forms in  papers of Evertse 
\cite{Eve1} and  Bennett \cite{Ben2}. These analytic  tools and their adjustments   are developed in Sections \ref{Pade}, \ref{AN} and \ref{AAL}. In the remaining final sections we complete our proofs.

\section{Siegel's Theorem on diagonalizable forms}\label{compareSiegel}

In \cite{Sie3} Siegel proved the following theorem.
	\begin{thm}[Siegel]\label{siegels}
	Assume that $F(x , y)$ is a diagonalizable form of degree $r$ and with discriminant $\Delta$. Suppose that
	\begin{equation}\label{siedel}
		\Delta' > \left(r^4 h\right)^{c_{l} r^{2-l}}
	\end{equation}
	where $r \geq 6-l,l=1,2,3,$
	\begin{equation}\label{c1c2c3}
c_1=45+\frac{593}{913},\ c_2=6+\frac{134}{4583}\textrm{ and } c_3=75+\frac{156}{167},
\end{equation}
and $\Delta'$ is defined in \eqref{Deltaprime}.
	Then
	\begin{equation}\label{NFh}
		N_F(h)\leq \begin{cases} 2lr &\mbox{if } D<0 \\ 
	4l & \mbox{if } D>0, \ r \textrm{ is even and $F$ is indefinite}\\
2l & \mbox{if } D>0, \ r \textrm{ is odd and $F$ is indefinite}\\
1 & \mbox{if } D>0 \textrm{ and $F$ is definite}.
\end{cases} 
		\end{equation}
	\end{thm}

\noindent In particular, if $D<0$ and $l=1,$ then $N_F(h)\leq 2r$ provided
$$|\Delta|>2^{r^2-r}r^{183.6r}h^{47.6r-2}.$$

Notice that in Theorem \ref{2r1}, the lower bound for $\Delta'$ is asymptotic to $r^{13r}h^{4r}$. In Theorem \ref{thml},
$$\alpha_1\sim\frac{7}{r^{m-4}}\ \ \textrm{and}\ \ \alpha_2\sim\frac{1}{r^{m-4}}.$$

To compare Theorem \ref{thml} with Theorem \ref{siegels}, we take $m=2l$ with $l=2,3$. Then the bounds for $N_F(h)$ in \eqref{NFh} and \eqref{kr} coincide. Corresponding to $\alpha_1$ and $\alpha_2$ in \eqref{Delta_condn}, we have $4c_{l} r^{2-l}$ and $c_{l} r^{2-l}$, respectively in \eqref{siedel}. Table \ref{table} provides the values of these quantities.

	\begin{table}[ht]
	\caption{Comparison of Theorem \ref{thml} with Siegel's Theorem } 
	\centering 
	\begin{tabular}{c c c } 
	\hline
	 & $\alpha_1$ &$\alpha_2$ \\ [0.15cm]
	\hline \hline 
	$m=4$ & & \\[0.05cm]
	\hline\hline\\ [-0.15cm]
	Theorem \ref{siegels} & $24+\frac{536}{4583}$ & $6+\frac{134}{4583}$\\[1.2ex]
	Theorem \ref{thml} & $7+\frac{7(2r^2-r+2)}{r^3-3r^2+r-2}$ & $1+\frac{3r^2}{r^3-3r^2+r-2}$\\ [0.15cm]
	\hline\hline 
	$m=6$ & & \\[0.05cm]
	\hline\hline \\ [-0.15cm]
	Theorem \ref{siegels} & $\frac{1}{r}(303+\frac{123}{167})$ & $\frac{1}{r}(75+\frac{156}{167})$\\[1.2ex]
	Theorem \ref{thml} & $\frac{7r^2(r-1)}{r^5-5r^4+10r^3-10r^2+3r-2}$ & $\frac{(r-1)(r^2+r+2)}{r^5-5r^4+10r^3-10r^2+3r-2}$\\[1.2ex] 
	\hline
	\end{tabular}
	\label{table} 
	\end{table}

\vskip 2mm

\noindent \textbf{Remark.} In \cite{Sie3}, Siegel considered the form
	$$
	 F(x , y) = (\alpha x + \beta y)^r + (\gamma x + \delta y)^r.
	 $$ 
Such a form can be represented as in \eqref{form}. Let $\omega$ be an $r$-th root of $-1$. Then
	\[
	F(x , y) = (\alpha x + \beta y)^r - (\gamma' x + \delta' y)^r,
	\]
where $\gamma'=\omega\gamma$ and $\delta'=\omega\delta$. Hence there
is no loss of generality in assuming that $F(x , y)$ is of the form
\eqref{form}.

\section{Preliminaries}\label{pre}
In this section we survey several facts about diagonalizable forms.  Most  of these facts can be found in \cite[p.148-149]{Sie3}.  Let $(x,y)$ be a generic primitive solution of \eqref{firstineq}. If $f$ is \textit{any} function of $(x,y)$, then we write
	\[
	f=f(x,y).
	\]
While enumerating the solutions of \eqref{firstineq} as $(x_0,y_0)$, $(x_1,y_1)$, $\ldots$, we denote by
	\[
	f_i=f(x_i,y_i),\ i\geq 0.
	\]
Let
	$$
	u=\alpha x + \beta y,\ v=\gamma x + \delta y,\ \xi = u^r \textrm{ and } \eta=v^r. 
	$$
Then $F(x , y) = \xi  - \eta=u^r-v^r.$ For any solution $(x,y)$, we have $(u(x,y),v(x,y))\neq(0,0)$ as $|F(x,y)|>0.$
 The Hessian $H$ and the Jacobian $P$ of $F$ are defined as 
	 $$
	 H=H_F(x,y) = F_{xx} F_{yy} - F^{2}_{xy}
	 $$
 and
	 $$
	 P=P_F(x,y) = F_{x} H_{y} - F_{y} H_{x},
	 $$
 respectively. It can be seen that
 \begin{equation}\label{deltainrj}
	 \Delta = (-1)^{\frac{(r-1)(r+2)}{2}} r^r j^{r(r-1)},
	 \end{equation}
	  \begin{equation}\label{Hinrj}
	 H = -r^2 (r -1)^2 j^2  (uv)^{r-2}   
 \end{equation}
and
	 \begin{equation}\label{Pinrj}
 P = -r^3 (r -1)^2 (r -2) j^3  (uv)^{r-3} (\xi + \eta).
 \end{equation}
Clearly if the coefficients of $F$ are all integers then $\Delta$ is
an integer and the coefficients of $H$ and $P$ are also integers. From \eqref{chiquad}, we have
	\begin{equation}\label{A,CinZ}
	\frac{\alpha \gamma}{\chi} = A \in \mathbb{Z}, \, \, \frac{\beta \delta}{\chi}= C \in \mathbb{Z},
\end{equation}
	$$
	\frac{1}{\chi} \left(\alpha \delta +  \beta \gamma\right) = B \in \mathbb{Z}.
	$$
We also have
	$$
	\frac{j}{\chi} = \frac{\alpha \delta -  \beta \gamma}{\chi} = \pm \sqrt{D} = d_{1}.
	$$
Therefore
	\begin{equation}\label{Bd1}
	\frac{\alpha \delta}{\chi}  = \frac{1}{2} \left(B + d_{1}   \right) \, \,   \textrm{and}\,   \frac{\beta \gamma }{\chi}  = \frac{1}{2} \left(B - d_{1}   \right).
	\end{equation}
A result of Gauss (see \cite{Sie3} for details) implies that 
	\begin{equation}\label{LGauss}
	L = r^2 (r - 1)^2 j^2 \chi^{r-2} = r^2 (r-1)^2 \chi^r D \in \mathbb{Z}
	\end{equation}
and hence
	\begin{equation}\label{chirat}
	\chi^{r} \in \mathbb{Q}.
\end{equation}
Now from \eqref{Pinrj} we have
	$$
	r^3 (r -1)^2 (r -2) \chi^{r} D   (Ax^2 + B xy+ Cy^2)^{r-3}  \sqrt{D}(\xi + \eta)\in \mathbb Z[x,y].
$$
Thus we get the following lemma.
\begin{lemma}\label{rationalcoeff}
The  binary form $\sqrt{D} (\xi + \eta)$ in $x$, $y$ has rational coefficients.
\end{lemma}

Let 
	$$
	Q(x , y) = \sqrt{D} (\xi + \eta).
	$$
We have
	\begin{equation}\label{conjugates}
	\xi(x , y) = \frac{F(x , y)}{2} + \frac{Q(x , y)}{2\sqrt{D}}\, \, \textrm{and}\, \, \eta(x , y) = -\frac{F(x , y)}{2} + \frac{Q(x , y)}{2\sqrt{D}}.
	\end{equation}
Now we will consider the number field $\mathbb{Q}(\sqrt{D})$.
If $\mathbb{Q}(\sqrt{D}) = \mathbb{Q}$ then both $\xi = (\alpha x + \beta y)^r$ and  $\eta= (\gamma x + \delta y)^r$ have rational coefficients. Furthermore, if $\mathbb{Q}(\sqrt{D}) = \mathbb{Q}$, after a change of variable, we may assume that  $\alpha \gamma \neq 0$.
If $\mathbb{Q}(\sqrt{D}) \neq \mathbb{Q}$, the corresponding coefficients of the forms $\xi$ and $-\eta$ are  conjugates in $\mathbb{Q}(\sqrt{D})$ by \eqref{conjugates}.  Therefore by \eqref{integralcondition}, we conclude that if $\mathbb{Q}(\sqrt{D}) \neq \mathbb{Q}$, then $\alpha \neq 0$ and $\gamma \neq 0$.  We may write
	\begin{equation}\label{SiegelE22}
	\xi = \alpha_{1}(x + \beta_{1} y)^r, \, \eta = \gamma_{1}(x + \delta_{1} y)^r
	\end{equation}
with 
	$$
	\alpha_{1} = \alpha^{r},\,  \gamma_{1} = \gamma^{r}, \, \beta_{1} = \frac{\beta}{\alpha}, \, \delta_{1} = \frac{\delta}{\gamma},\, \delta_{1} - \beta_{1} = \frac{j}{\alpha \gamma} \neq 0.
	$$
Then $\alpha_{1}$, $-\gamma_{1}$ and $\beta_{1}$, $\delta_{1}$ are either all rational numbers or pairs of algebraic conjugates in $\mathbb{Q}(\sqrt{D})$. Thus when $D<0$, we have
	\[
	|\xi|=|\eta| \textrm{ and } |u|=|v|.
	\]
	
	Throughout the rest of this manuscript, we may  assume, without loss of generality, that  $\alpha \gamma \neq 0$. 	
		
Let $\mathcal{O}$ be the ring of integers in $\mathbb{Q}(\sqrt{D})$. 
	\begin{lemma}\label{Siegelshows}
	All the  coefficients of $r (r-1) \sqrt{D} \xi$ and 
$r(r -1) \sqrt{D}\eta$ are in $\mathcal{O}$. 
	\end{lemma}
	\begin{proof}
	We have
	$$
	\xi \eta =(uv)^r=  (\alpha x + \beta y)^r  (\gamma x + \delta y)^r  = \chi^r (A x^2 + B xy + Cy^2)^r
	$$
and therefore
	\begin{equation}\label{rLD}
	r^2 (r-1)^2 D \xi \eta = L (Ax^2 + B x y + C y^2)^r \in \mathbb{Z}[x , y].
	\end{equation}
	Since $\xi - \eta \in \mathbb{Z}[x , y]$ and
		$$
		\frac{r(r-1)}{2} \sqrt{D} (\xi + \eta) = \left(  \left( \frac{r(r-1)}{2} \right)^2 D (\xi - \eta)^2 + r^2(r -1)^2 D \xi \eta \right)^{1/2},
		$$
	we conclude that all of the coefficients of  
$r (r-1) \sqrt{D} \xi$ and $r(r -1) \sqrt{D}\eta$ are in $\mathcal{O}.$
	\end{proof}
	 \begin{lemma}\label{122Diag}
	 We have
		   \begin{equation*}
		   \frac{u( x , y)}{u(1 , 0)}  , \frac{v( x , y)}{v(1 , 0)} \in \mathbb{Q}(\sqrt{D})[x , y].
		   \end{equation*}
	\end{lemma}
 \begin{proof}
We have $\frac{\beta}{\alpha}, \frac{\delta}{\gamma} \in \mathbb{Q}(\sqrt{D})$. Therefore
	 \begin{equation*}
	 \frac{u( x , y)}{u(1 , 0)}  , \frac{v(x , y)}{v(1 , 0)} \in \mathbb{Q}(\sqrt{D})[x , y].
	 \end{equation*}
 \end{proof}
	  \begin{lemma}\label{ai2Diag}
	 Let $(x_{1} , y_{1})$ and $(x_{2}, y_{2})$ be two pairs of rational integers. Then
		 $$
		\frac{2}{\chi} \, u_1 v_2 \in\mathcal{O}. 
		 $$
	Let $a$, $b$ be positive integers with $a+b=r$. Then
		  $$
		  {r\choose a} r(r-1) \sqrt{D} \, u_1^a u_2^b\ \textrm{ and }\ {r\choose a} r(r-1) \sqrt{D}\,  v_1^a v_2^b
		 $$
	   are in $\mathcal{O}$.
	  \end{lemma} 
	   \begin{proof}
The first assertion follows easily	 
by \eqref{A,CinZ} and \eqref{Bd1}.
Since 
$$ r(r-1) \sqrt{D}u^{r} = r(r-1)\sqrt{D} \xi,$$
 from Lemma 
\ref{Siegelshows} we conclude that   \\
$ \begin{pmatrix} r\\ k \end{pmatrix} r(r-1)\sqrt{D}\alpha^{k} \beta^{r-k}$, 
for $k =0, \ldots, r$,  are all in   $\mathcal{O}$. 
The coefficients of $u_{1}^a u_{2}^b$  are of the form
$${a\choose i}{b\choose j}{\alpha^{i+j} \beta^{r-i-j}}$$
which is equal to
$$ \frac{{a\choose i}{b\choose j}}{{r\choose i+j}}{r\choose i+j}
{\alpha^{i+j} \beta^{r-i-j}}
= \frac{{i+j\choose i}{r-i-j\choose a-i}}{{r\choose a}} {r\choose i+j}
{\alpha^{i+j} \beta^{r-i-j}}.$$
Hence we obtain that  
${r\choose a} r(r-1) \sqrt{D} u_1^au_2^b$ is in $\mathcal{O}$.
We can similarly show that ${r\choose a} r(r-1)\sqrt{D}v_1^a v_2^b$ 
is also in $\mathcal{O}$.
	\end{proof}

\section{Reduced Forms}\label{EF}
We call two binary forms $F_{1}(x , y)$ and $F_{2}(x , y)$ equivalent if they are equivalent under $\textrm{GL}(2,\mathbb{Z})$ action, i.e., if there exists an integer matrix 
	\[
	\lambda=\begin{pmatrix} a & b \\ c & d \end{pmatrix}
	\]
such that $ad - bc = \pm 1$ and 
	$$
	F_{1}(ax + by, cx + dy) = F_{2} (x , y).
	$$
Then we write
	$$
	F_{2} = F_{1} \circ \lambda.
	$$
Notice that if
	$$
	F( x , y) = (\alpha x + \beta y)^r - (\gamma x + \delta y)^r
	$$
 then every equivalent form will be of the shape
	$$
	G (x , y) = \left( (a \alpha + c \beta) x + (b \alpha + d \beta) y\right)^r - \left((a \gamma + c \delta) x + (b \gamma + d \delta) y   \right)^r,
	$$
for some $\begin{pmatrix} a & b \\ c & d \end{pmatrix}\in$ $\textrm{GL}(2,\mathbb{Z})$. Thus a diagonalizable form remains diagonalizable under $\textrm{GL}(2,\mathbb{Z})$-action. 
Further, the number of primitive solutions of $0<|F_{1}(x , y)|\leq h$ remains unaltered for any equivalent form $F_{1}$ of $F$.

Recall that a definite quadratic form $Ax^2 + B xy + Cy^2$ is called reduced if $C \geq A \geq |B|$. 

\begin{defin} Let  $F(x , y) = (\alpha x + \beta y)^r - (\gamma x + \delta y)^r \in \mathbb{Z}[x , y]$ be a diagonalizable form with  
	\begin{equation}
(\alpha x + \beta y)  (\gamma x + \delta y) = \chi (Ax^2 + B xy+ Cy^2)
\end{equation}
and $A, B, C \in \mathbb{Z}$ and a constant $\chi$.  If 
$D = B^2 - 4 A C < 0$, we call $F(x, y)$ \textit{reduced} if the quadratic form $Ax^2 + B xy+ Cy^2$ is reduced.
\end{defin}

It is a well-known fact that every definite quadratic form is equivalent to a reduced form. Therefore, 
 if $D < 0$ then $F(x, y)$ is clearly equivalent to a reduced form.

\begin{lemma}\label{negativedisc}
	Let $F(x , y)$ be a reduced diagonalizable form  with $D<0$. If $(x,y)$ is a solution of $0 < |F(x , y)| \leq h$ with $y\neq 0$, then
		\begin{equation}\label{uvyEX}
		|u(x,y)| = |v(x,y)| \geq \frac{|\chi|^{1/2} |y| |3D|^{1/4}}{2}.
		\end{equation}
	\end{lemma}
\begin{proof}
We have
	$$
A x^2 + B xy + Cy^2 = y^2 (A t^2 + B t + C),
	$$
where $t = x/y$. Then the polynomial $A t^2 + B t + C$ assumes a minimum equal to 
$\frac{4AC - B^2}{4A}$ at $t = \frac{-B}{2A}$. Since $Ax^2 + B xy + Cy^2$ is reduced, we have $C \geq A \geq |B|$, and therefore
	$$
	4A^2 \leq 4 A C
	$$
and
	$$
	A^2 \geq B^2.
	$$
We conclude that 
	$$
	3A^2 \leq4AC - B^2
	$$
and
	$$
	A  \leq \frac{\sqrt{-D}}{\sqrt{3}}.
	$$
Therefore
	$$
	\frac{4AC - B^2}{4A} =\frac{-D}{4A}\geq\frac{\sqrt{-3D}}{4}.
	$$
This implies that 
	\begin{equation}\label{ABCy}
	A x^2 + B xy + Cy^2  \geq y^2  \frac{\sqrt{-3D}}{4}.
	\end{equation}
Thus if $D<0$, by \eqref{chiquad}, we have
	\begin{equation*}
	|u| = |v| \geq \frac{|\chi|^{1/2} |y| |3D|^{1/4}}{2}.
	\end{equation*}
\end{proof}
\section{Gap Principles}\label{GP}
 We define
	\begin{equation}\label{defofmu}
	\mu(x , y) : = \frac{\eta(x , y)}{\xi(x , y)}
	\end{equation}
so that
	\begin{equation}\label{muF}
	1 - \mu(x , y) = \frac{\xi(x , y)-\eta(x , y)}{\xi(x , y)} = \frac{F(x , y)}{\xi(x , y)}.
	\end{equation}
	Let $(x,y) \in \mathbb{Z}^2$ satisfy the inequality $0 < |F(x , y)| \leq h$. We define 
	 \begin{equation}\label{defofZ}
Z=Z(x,y): =\max{(|u(x , y)|,|v(x , y)|)}
	\end{equation}
and
	 \begin{equation}\label{defofzeta}
	\zeta = \zeta (x , y):=\frac{|F(x , y)|}{Z^r(x , y)}.
	\end{equation}
Then
	\begin{eqnarray}\label{SE31}
	\max (|\xi(x , y)| , |\eta(x , y)|)& = & |F(x , y)| \zeta^{-1}(x , y), \\ \nonumber
	 Z& = & |F(x , y)|^{1/r} \zeta(x , y)^{-1/r}.
	\end{eqnarray}
For brevity, we set
	\[
	\zeta_i=\zeta(x_i,y_i), \ \mu_i=\mu(x_i,y_i) \textrm{ and } Z_i=\max(|u(x_i,y_i)|,|v(x_i,y_i)|).
	\]
Write
$$F(x,y)=\xi(x , y)-\eta(x , y) = \prod_{k=1}^r (u(x , y)-v(x , y)e^{\frac{2\pi ik}{r}}).$$

\begin{defin}\label{defofrelated}
Let $\omega$ be an $r$-th root of unity. We say that $(x,y)$ is {\it
  related} to $\omega$ if 
$$|u(x , y)-v(x, y)\omega|=\min_{1 \leq k \leq r}|u(x , y)-v(x , y)e^{\frac{2\pi i k}{r}}|.$$
\end{defin}

\begin{defin}\label{defofS}
We denote by $S$ the set of all solutions of $0< |F(x , y| \leq h$ and by
$S_\omega$ the set of all solutions of $0< |F(x , y| \leq h$ that are related to $\omega$. 
\end{defin}

Clearly $S_\omega\subseteq S$ and $S = \cup S_{\omega}$, as $\omega$ ranges over all $r$-th roots of unity. 
	\begin{lemma}\label{Except1}
	Let $F(x , y)$ be a diagonalizable form and $(x_0,y_0)$ a
        solution to the inequality $0 < |F(x , y)| \leq h$ with the largest value $\zeta_{0}$ of
        $\zeta$, where $\zeta$ and $Z$  are defined in \eqref{defofzeta} and \eqref{defofZ}. Then for every integer pair $(x , y) \neq (x_{0}, y_{0})$ satisfying $0 < |F(x , y)| \leq h$, we have
	\begin{enumerate}[(i)]
	\item $Z(x , y)\geq\frac{|j|^{1/2}}{2^{1/2}h^{1/r}}.$
	\item $Z(x , y)\geq\frac{|j|}{2h^{1/r}}$ if $\zeta_0 \geq 1$.
	\item $\zeta(x , y)<2^{-\nu}$ if $\zeta_0 \geq 1$ and $|j|>2^{1+\nu/r}h^{2/r}$ for $\nu\in\mathbb{R}$.
	\end{enumerate}	
	\end{lemma}
\begin{proof}
Let $(x_1,y_1)\neq(x_0,y_0)$ be a solution of \eqref{firstineq}. Then
	$$
	u_{0} v_{1} - u_{1} v_{0} = (\alpha \delta -\beta \gamma) (x_{0} y_{1} - x_{1} y_{0}) = j (x_{0} y_{1} - x_{1} y_{0}) \neq 0
$$
by \eqref{integralcondition}.
We conclude that
	\begin{equation}\label{SE34}
	|j| \leq |u_{0} v_{1}| + |u_{1} v_{0}| \leq 2Z_0 Z_1.
	\end{equation}
	\begin{enumerate}[(i)]
	\item By \eqref{SE34} and \eqref{SE31}, we have
	\[
	|j|\leq 2Z_0 Z_1\leq 2h^{2/r}\zeta_0^{-1/r}\zeta_1^{-1/r}\leq 2h^{2/r}\zeta_1^{-2/r}.
	\]
	Hence
	\[
	Z_1^{-r}\leq \zeta_1\leq 2^{r/2}h|j|^{-r/2},
	\]
	proving the claim.
	\item From \eqref{SE34}, we get
	\[
	Z_1\geq \frac{|j|\zeta_0^{1/r}}{2|F_0|^{1/r}}\geq
\frac{|j|}{2h^{1/r}} \textrm{ if } \zeta_0\geq 1.
\]
	\item If $\zeta_0 \geq 1$, then
	\[
	|j|\leq 2Z_0 Z_1\leq 2h^{2/r}\zeta_1^{-1/r}.
	\]
Hence, if $|j|>2^{1+\nu/r}h^{2/r}$, we get
	\[
	\zeta_1\leq 2^rh^2|j|^{-r}<2^{-\nu}.
	\]
	\end{enumerate}
\end{proof}

\begin{defin}\label{defofx,y0}
We  denote the solution to the inequality $0 < |F(x, y)|\leq h$ for which 
$\zeta$ is the largest by $(x_0,y_0)$. We denote the largest value of $\zeta$ by $\zeta_{0}$.
\end{defin}

\bigskip

\noindent{\bf Remark.} By Lemma \ref{Except1}, if
$|j| >2h^{2/r}$  and the integer pair $(x , y) \neq (x_{0}, y_{0})$ satisfies $0 < |F(x , y)| \leq h$, then  
$\zeta(x , y) <1$.

\bigskip

The next three results are for forms
with $D>0.$ 
\begin{lemma}\label{definiteforms}
Let $F$  be a definite diagonalizable  form with $D>0.$
Then for any $(x,y)\in \mathbb R^2\setminus (0,0),$ we have 
$\zeta(x , y)\geq 1$, where the function $\zeta$ is defined in \eqref{defofzeta}.
\end{lemma}
\begin{proof}
Since $F(x,y)=(-1)^rF(-x,-y),$ we conclude that the degree $r$ is even. Further,
since $j\neq 0,$ forms of the type $ax^r$ or $by^r$ with $a,b\in \mathbb Z$
are excluded.
It follows from the definition of $\zeta$ in \eqref{defofzeta} that if  $\alpha^r$ and $\gamma^r$ are of opposite signs, then  $\zeta(x , y) \geq 1$ for
every $(x,y)\in \mathbb{R}^2$. We claim that for every definite  diagonalizable  form $F(x , y)$, $\alpha^r$ and $\gamma^r$  have opposite  signs. 
Assume, in contrary, that  $\alpha^r$ and $\gamma^r$  are either both positive or both negative. Recall that (see \eqref{SiegelE22} and its following lines)
$$F(x,y)=\alpha^r(x+\beta_1y)^r-\gamma^r(x+\delta_1y)^r$$
with $\alpha \gamma \neq 0$ and $\alpha^r,\gamma^r,\beta_1,\delta_1 \in
\mathbb Q(\sqrt{D}) \subseteq \mathbb{R}$.  Then we have
$$
F(-\beta_{1} , 1) = -\gamma^r(-\beta_{1}+\delta_1)^r
$$
and 
$$
F(-\delta_{1}, 1) = \alpha^r(-\delta_{1}+\beta_1)^r.
$$
Since $-\beta_{1}+\delta_1 \in \mathbb{R}$ and $r$ is even, we conclude that $F(-\beta_{1} , 1)$  and 
$F(-\delta_{1}, 1)$ have opposite signs, which is a contradiction with the form $F$ being definite. 
\end{proof}
As a direct consequence of Lemmas \ref{Except1}(iii) and
\ref{definiteforms} we get
\begin{cor}\label{definite} 
Let $F(x , y)$ be a definite  diagonalizable  form with
  $D>0$ and $|j|> 2h^{2/r}.$
Then the inequality $0 < |F(x , y)| \leq h$ has at most one solution.
\end{cor}
\begin{lemma}\label{positiveDforms}
Let $F(x , y)$ be a diagonalizable  form with
  $D>0$. Then all solutions of $0 < |F(x , y)| \leq h$ with  $\zeta<1$
are related to one or two $r$-th roots of unity when $r$ is odd
or $r$ even, respectively.
\end{lemma}
\begin{proof}
Let
$u=\alpha(x+\beta_1y)$ and $v=\gamma(x+\delta_1y)$ so that
$$F(x , y) = u^r(x , y) - v^r(x, y),$$
where $\alpha \gamma \neq 0$ and $\alpha^r, \gamma^r, \beta_1,\delta_1 \in \mathbb
Q(\sqrt{D})$. 
It follows from the definition of $\zeta$ in \eqref{defofzeta} that if  $r$ is even and $\alpha^r$ and $\gamma^r$ are of opposite signs, then  $\zeta(x , y) \geq 1$ for
every $(x,y)\in \mathbb{R}^2$ (in this case $F(x,y)$ will be a definite form, see Lemma \ref{definiteforms} and its proof). Therefore we will assume that either $r$ is odd or the real numbers $\alpha^r$ and $\gamma^r$  have the same sign. Without loss of generality, let us assume that $\alpha^r$ and $\gamma^r$  are both positive if $r$ is even (otherwise we can replace the form $F(x , y)$ by $-F(x , y)$). Let $\alpha'$ and $\gamma'$ be fixed real $r$-th roots of $\alpha^r$ and $\gamma^r$, respectively. Then we have $\frac{\alpha}{\gamma} = \omega \frac{\alpha'}{\gamma'}$ for a fixed $r$-th root of unity $\omega$. Therefore, for $(x , y) \in \mathbb{R}^2$,
$$
\frac{u(x , y)}{v(x , y)} =  \omega R(x , y),
$$
with $R(x , y) \in \mathbb{R}$.
We claim that all the solutions of  $0 < |F(x , y)| \leq h$ with  $\zeta<1$ are related to $\omega$ if $r$ is odd and all the solutions of  $0 < |F(x , y)| \leq h$ with  $\zeta<1$ are related to $\omega$  or $-\omega$ if $r$ is even.

Suppose  $(x , y) \in \mathbb{Z}^2$ satisfies $0 < |F(x , y)| \leq h$ and $\zeta(x , y) < 1$. Therefore, by \eqref{defofzeta}, we have $u(x , y) \neq 0$ and $v(x , y) \neq 0$.
Let $\omega_1$ be an $r$-th root of unity.
We have
\begin{equation}\label{uvomega}
\left|\frac{u(x , y)}{v(x , y)}-\omega_1\right| = \left|\frac{u(x , y)}{v(x , y)}\omega^{-1}-
\omega_1 \omega^{-1}\right|.
\end{equation}
Since $\frac{u(x , y)}{v(x , y)}\omega^{-1} \in \mathbb{R}$, we conclude that 
$$\left|\frac{u(x , y)}{v(x , y)}\omega^{-1}-
\omega_1 \omega^{-1}\right|
\geq 
\left|\frac{u(x , y)}{v(x , y)}\omega^{-1}-1\right| = \left|\frac{u(x , y)}{v(x , y)}-\omega\right|$$
if $\frac{u(x , y)}{v(x , y)}\omega^{-1} > 0$,
and 
$$\left|\frac{u(x , y)}{v(x , y)}\omega^{-1}-
\omega_1 \omega^{-1}\right|
\geq 
\left|\frac{u(x , y)}{v(x , y)}\omega^{-1}+1\right| =  \left|\frac{u(x , y)}{v(x , y)}+\omega\right|$$
if $\frac{u(x , y)}{v(x , y)}\omega^{-1} <0$. Therefore by \eqref{uvomega} and Definition \ref{defofrelated}, we conclude that  the solution $(x , y)$ is related to $\omega$ if $r$ is odd and $(x , y)$ is related to $\omega$ or $-\omega$ if $r$ is even. Notice that if $\zeta(x , y) < 1$, then  the real numbers $u^r(x , y)$ and $v^r(x , y)$ have the same sign and therefore when $r$ is odd, the real number $\frac{u(x , y)}{v(x , y)}\omega^{-1} > 0$.
\end{proof}
	\begin{lemma}\label{6.12}
Suppose $F$ is a diagonalizable binary form of degree $r$.	
Let $(x,y)$ be a solution of $0 < |F(x , y)| \leq h$
related to a fixed $r$-th root of unity, say $\omega$. Then 
	 \begin{equation}\label{Gap12}
		\left|\omega  - \frac{u(x , y)}{v(x , y)}\right| \leq \frac{\pi}{2r} \zeta(x , y)\ if D<0.
		\end{equation}
	If further $\zeta < 1$ and $D<0,$ then 
		\begin{equation}\label{Gap22}
		\left|\omega - \frac{u(x , y)}{v(x , y)}\right| < \frac{\pi}{3r} \zeta(x , y).
		\end{equation}
Suppose $D>0$ and $\zeta<1.$ Then 
	\begin{equation}\label{Gap23}
		\left|\omega - \frac{u(x , y)}{v(x , y)}\right| \leq \frac{Z(x , y)}{|v(x , y)|}\zeta(x , y),
		\end{equation}
		where
$Z(x,y)$ and $\zeta(x , y)$ are defined in \eqref{defofZ} and 
\eqref{defofzeta}, respectively.
	\end{lemma}
	\begin{proof}
	Let $D<0$. Let
		\[
		\theta=\textrm{arg} \left( \frac{u(x , y)}{\omega v(x , y)}\right),
		\]
		so that $-\pi < \theta \leq \pi$.
	Since 
		\[
		\left|\frac{u(x , y)}{v(x , y)}\right|=1
		\]
	and $(x,y)$ is related to $\omega$, we have
	 $$
	 r \theta = \textrm{arg} \left( \frac{u^r(x , y)}{v^r(x , y)}\right) =  
\textrm{arg} \left( \frac{\xi(x , y)}{\eta(x, y)}\right).
	  $$
Thus
	$$|\theta| \leq \frac{\pi}{r}. $$
Also
	  $$
	  \sqrt{2 -2\cos(r\theta)} = \zeta(x , y)
	  $$
gives
$$|\theta| < \frac{\pi}{3r},$$
whenever $\zeta<1.$
Further since 
	 $$ 
	\left|\omega - \frac{u(x , y)}{v(x , y)} \right| \leq  |\theta|, 
	 $$
	we obtain
	$$
	\left|\omega - \frac{u(x , y)}{v(x , y)}\right| \leq \frac{1}{r} 
\frac{|r\theta|}{\sqrt{2 -2\cos(r\theta)}} 
	\left|1 -\frac{u^r(x , y)}{v^r(x , y)}\right|. 
	$$
	By differential calculus 
$\frac{|r\theta|}{\sqrt{2 - 2\cos(r\theta)}} \leq \frac{\pi}{2}$ 
whenever $ 0<|\theta| \leq \frac{\pi}{r} $. Therefore
	 $$
	 \left|\omega - \frac{u(x , y)}{v(x , y)}\right| \leq\frac{\pi}{2 r} \zeta(x , y),
	 $$
	 and from the fact that $\frac{|r\theta|}{\sqrt{2 - 2\cos(r\theta)}} 
< \frac{\pi}{3}$ whenever $ 0<|\theta| < \frac{\pi}{3 r} $ , we conclude
	 $$
	 \left|\omega - \frac{u(x , y)}{v(x , y)}\right| <\frac{\pi}{3 r} \zeta(x , y),
	 $$
	as desired. 	
	
Now assume that  $D>0.$ Let $(x , y)$ be a solution to the inequality $0<|F(x,y)|\leq h$ with $\zeta<1$, and $\mu^{-1/r}$ be the positive real $r$-th root of $\mu^{-1} = \frac{\xi(x , y)}{\eta(x , y)} $. Then 
	\[
	u/v=\mu^{-1/r}\omega
	\]
for some $r$-th root of unity $\omega$.
Let $|v| \geq |u|.$ Then
$$\left|\frac{u}{v}-\omega\right|=|\mu^{-1/r}-1|=1-\mu^{-1/r}\leq 1-\mu^{-1}=\zeta.$$	
Suppose $|v|<|u|.$ Then
\begin{eqnarray*}
&&\left|\frac{u}{v}-\omega\right|=\mu^{-1/r}-1\\
&=&\mu^{-1/r}(1-\mu^{1/r})\leq\mu^{-1/r}(1-\mu)=\left|\frac{u}{v}\right|\left(1-\frac{v^r}{u^r}\right)=\frac{Z}{|v|}\zeta.
\end{eqnarray*}
\end{proof}

Let $\omega$ be a fixed $r$-th root of unity. As before, we will denote the set of all solutions of the inequality $0 < |F(x , y)| \leq h$  which are related to $\omega$ by $S_{\omega}$. Assume  $S_{\omega} \neq \emptyset$ and let

	\[
	\zeta'=\max\limits_{(x,y)\in S_{\omega}} \zeta(x,y)
	\]
and  $(x',y')$ the solution in $S_{\omega}$ with $\zeta(x',y')=\zeta'$. We define
\begin{equation}\label{defofS'}
S'_{\omega}: =S_{\omega} \setminus \{(x', y')\}.
\end{equation}

 \begin{defin} \label{defofR(k)} Let $k$ be a positive integer. We define
$$R(k): =(r-1)^{k-1}.$$\end{defin}
\begin{lemma}\label{43}
Let $S'_\omega$ be given by \eqref{defofS'}. Assume that  $|S'_\omega| \geq 2.$ Then 
	\[
Z(x,y)\geq\frac{|j|}{2h^{1/r}} \textrm{ for all } (x,y)\in S'_\omega.
	\]
Moreover,  if  $(x_{j_1},y_{j_1})$, $\ldots$, $(x_{j_t},y_{j_t})\in S_\omega$ with $t \geq 3$, 
$\zeta_{j_t}\leq\ldots\leq\zeta_{j_1} < 1$, and 
	\[
	|j|>2^{1+(r-2)/(r(R(t-1)-1))}h^{2/r},
	\]
then $\zeta_{j_{t-1}}<1/2$.
\end{lemma}
\begin{proof}
Let $(x_{0}, y_{0})$ be as in Definition \ref{defofx,y0}. If $\zeta_{0}\geq 1$, then  the first part of the result follows from Lemma
\ref{Except1}(ii). Assume that $\zeta_{0}< 1$. Then $\zeta(x , y)<1$ for every $(x , y) \in S_{w}$. 
Let $(x_{i_0},y_{i_0})$, $(x_{i_1},y_{i_1})$ $\in S_\omega$ with $\zeta_{i_0}\geq\zeta_{i_1}$.
We have
	$$
	u_{i_0} v_{i_1} - u_{i_1} v_{i_0} = (\alpha \delta -\beta \gamma) (x_{i_0} y_{i_1} - x_{i_1} y_{i_0}) = j (x_{i_0} y_{i_1} - x_{i_1} y_{i_0}) \neq 0.
	$$
Using Lemma \ref{6.12}, we obtain
	\begin{eqnarray}
\nonumber	|j|&\leq& |v_{i_0}v_{i_1}|\left|\frac{u_{i_0}}{v_{i_0}}-\frac{u_{i_1}}{v_{i_1}}\right|\leq|v_{i_0}v_{i_1}|\left(\left|\frac{u_{i_0}}{v_{i_0}}-\omega\right|+\left|\frac{u_{i_1}}{v_{i_1}}-\omega\right|\right)\\
\nonumber	&\leq& |v_{i_0}v_{i_1}|\left(\frac{Z_{i_0}\zeta_{i_0}}{|v_{i_0}|}+\frac{Z_{i_1}\zeta_{i_1}}{|v_{i_1}|}\right)\leq 2Z_{i_0}Z_{i_1}\zeta_{i_0}.
	\end{eqnarray}
Hence
	\[
	|j|\leq 2Z_{i_0}Z_{i_1}\zeta_{i_0}\leq 2h^{1/r}\zeta_{i_0}^{(r-1)/r}Z_{i_1} 
\leq 2h^{1/r}Z_{i_1},
	\]
proving the first part of the lemma. 

Now we assume  $(x_{j_1},y_{j_1})$, $\ldots$, $(x_{j_t},y_{j_t})\in S_\omega$, with $\zeta_{j_t}\leq\ldots\leq\zeta_{j_1} <1$. Since
	\[
	|j|\leq  2h^{1/r}\zeta_{j_1}^{(r-1)/r}Z_{j_2}\leq
        2h^{2/r}\zeta_{j_1}^{(r-1)/r}
\zeta_{j_2}^{-1/r},
	\]
we get
	\[
	\zeta_{j_2}\leq H^r \zeta_{j_1}^{r-1},
	\]
where
	\[
	H=2h^{2/r}|j|^{-1}.
	\]
Proceeding inductively, we obtain that
	\[
	\zeta_{j_{t-1}}\leq H^{r(R(t-1)-1)/(r-2)}\zeta_{j_1}^{R(t-1)}<H^{r(R(t-1)-1)/(r-2)}.
	\]
Thus $\zeta_{j_{t-1}}<1/2$ if 
	\[
	|j|^{r(R(t-1)-1)/(r-2)}>2^{1+r(R(t-1)-1)/(r-2)} h^{2(R(t-1)-1)/(r-2)}.
	\]
\end{proof}

Let $\omega$ be a fixed $r$-th root of unity.
 If  $0 \neq k=|S'_{\omega}|$, throughout this manuscript, we will index 
the elements  $(x_{1}, y_{1}), \ldots, (x_{k}, y_{k})$ of  $S'_{\omega}$  such that 
$\zeta_{i+1}\leq \zeta_{i}$ for $i=1, \ldots,  k-1$. The following result of  Siegel in  \cite[p. 154]{Sie3} provides an important gap principle. 
	\begin{lemma}\label{gap_Siegel}
Assume that $|S'_\omega| \geq 2.$
Let $i\geq 2$ be an integer. If $\zeta_{i-1}<1$,
then
		\[
		Z_i\geq \frac{|j|}{2h}Z_{i-1}^{r-1}.
		\]
	\end{lemma}
	\begin{proof}
	First we assume $D<0$. By \eqref{Gap22} and \eqref{SE31}, we have
	\begin{eqnarray*}
	|j| &\leq& |u_{i-1} v_{i} - u_{i} v_{i-1}| \leq |v_{i-1} v_{i}| 
\left( \left|\frac{u_{i-1}}{v_{i-1}}- \omega\right| + \left| \frac{u_{i}}{v_i}-\omega\right| \right)\\ &\leq& Z_{i-1}Z_i (\zeta_{i-1} + \zeta_{i})
	\leq 2Z_{i-1}Z_i \zeta_{i-1} \leq \frac{2hZ_i}{Z_{i-1}^{r-1}},
	\end{eqnarray*}
		which proves the lemma for $D<0$. 

Next we assume $D>0$. Then $\xi$ and $\eta$ are real. Suppose that  $\zeta < 1$. Then 
	$$
	\max(|\xi| , |\eta|) > |\xi - \eta| \textrm{ and  hence}\ \xi \eta > 0.
	$$
	We also have
	$$
	0 < \mu < 1 \, \, \textrm{if}\, \, |\eta| < |\xi|
	$$
	and
	$$
	0 < \mu^{-1} < 1 \, \,  \textrm{if}\, \, |\eta| > |\xi|,
	$$
	which gives either $\zeta = |1 - \mu|$ or $\zeta = |1 - \mu^{-1}|$.

	We may suppose without loss of generality, that $|\eta_{i-1}| < |\xi_{i-1}|$. The other case is similar. Then
	$$
	0 < \mu_{i-1} <1
	$$
	 and
	 $$
	 \mu_{i-1} < \mu_{i-1} ^{1/r} < 1
	 $$
	 where  $\mu_{i-1} ^{1/r}$ is the positive real $r$-th root of $\mu_{i-1}$.  
If $|\eta_{i}| < |\xi_{i}|$, we have
	 $$
	 0 < 1- \mu_{i}^{1/r} <1-\mu_i= \zeta_{i} \leq \zeta_{i-1}.
	 $$ 
	 This implies
	 $$
	 \left| (\mu_{i-1}^{1/r} -1) - (\mu_{i}^{1/r} -1)   \right| <2 \zeta_{i-1}.
	 $$
	Further,
	$$
	u_{i-1} v_{i} - u_{i} v_{i-1} = (\alpha \delta -\beta \gamma) (x_{i-1} y_{i} - x_{i} y_{i-1}) = j (x_{i-1} y_{i} - x_{i} y_{i-1}) \neq 0
	$$
	and  since every solution is related to a 
fixed root of unity, there exists an $r$-th root of unity, say $e$, such that
	\begin{eqnarray}
\nonumber	e (u_{i-1} v_{i} - u_{i} v_{i-1}) &=& u_{i-1} u_i \left( (1-\mu_{i-1}^{1/r} ) -
(1-\mu^{1/r}_{i})\right).
\end{eqnarray}
Hence
	\begin{eqnarray}\label{SE40}
|u_{i-1} v_{i} - u_{i} v_{i-1}|&\leq&2Z_{i-1}Z_i \zeta_{i-1}\leq \frac{2hZ_i}{Z_{i-1}^{r-1}}.
	\end{eqnarray}
	If, however, $|\xi_{i}| < |\eta_{i}|$, then we have
	$$
	0 < 1 - \mu_{i}^{-1} = \zeta_{i}, \, \, \,  0 < \mu_{i-1} \mu_{i}^{-1} < \mu_{i-1}^{1/r} \mu_{i}^{-1/r} < 1
	$$
	and 
	\begin{eqnarray*}
	0 & <& 1 - \mu_{i-1}^{1/r} \mu_{i}^{-1/r} <1-\mu_{i-1}\mu_i^{-1}\\
	&= & 1- (1-\zeta_{i-1}) (1 - \zeta_{i}) = \zeta_{i-1} + \zeta_{i} - \zeta_{i-1} \zeta_{i} < 2 \zeta_{i-1}.
	\end{eqnarray*}
	Since
	$$
	u_{i-1} v_{i}-v_{i-1}u_i = u_{i-1} v_{i}(1 - \mu_{i-1}^{1/r} \mu_{i}^{-1/r}),
	$$
	the lemma follows in this case as well.
	\end{proof}
\section{ Pad\'e Approximation}\label{Pade}
If
	$$
	P(z)=a_0z^d+a_1z^{d-1}+\ldots+a_d
	$$ 
is a polynomial of degree $d,$ define 
	$$
	P^{*}(x , y) =a_dx^d+\ldots+a_1xy^{d-1}+a_0y^d= x^{d} P(y/x).
	$$
A {\em hypergeometric function} is a power series of the form
	\begin{eqnarray*}
	&&F(\alpha , \beta , \gamma , z) =   \\
	&&1 + \sum_{k=1}^{\infty} \frac{\alpha (\alpha + 1) \cdots (\alpha + k -1) \beta (\beta + 1) \cdots (\beta + k - 1)}{\gamma (\gamma + 1) \cdots (\gamma + k - 1) k!} z^{k}.
	\end{eqnarray*}
Here $z$ is a complex variable and $\alpha$, $\beta$  and $\gamma$ are complex constants. If $\alpha$ or $\beta$ is a non-positive integer  and $m$ is the smallest integer such that 
	$$
	\alpha    (\alpha + 1) \cdots (\alpha + m) \beta (\beta + 1) \cdots (\beta + m) = 0,
	$$
then $F(\alpha , \beta , \gamma , z) $ is a polynomial in $z$ of degree $m$. Furthermore, if $\gamma$  is a non-positive integer, we will assume that at least one of $\alpha$ and $\beta$ is also a non-positive integer, greater than $\gamma$.

We note that $F( \alpha , \beta , \gamma , z)$ converges for $|z| < 1$.  By a result of Gauss, if $\alpha$, $\beta$  and $\gamma$ are real with $\gamma > \alpha + \beta $ and $\gamma$, $ \gamma - \alpha$ and $\gamma - \beta$ are not non-positive integers, then $F(\alpha , \beta , \gamma , z)$ converges for $z = 1$ and we have
	\begin{equation}\label{Gam}
	F(\alpha , \beta , \gamma , 1) = \frac{\Gamma(\gamma) \Gamma(\gamma - \alpha - \beta)}{\Gamma(\gamma - \alpha) \Gamma(\gamma - \beta)}.
	\end{equation} 
The hypergeometric function $F( \alpha , \beta , \gamma , z)$ satisfies the following second order differential equation 
	 \begin{equation} \label{hip}
	 z (1 - z) \frac{d^{2}F}{dz^{2}} + (\gamma - (1 + \alpha + \beta ) z )\frac{dF}{dz} - \alpha \beta F = 0.
	 \end{equation}

The following lemma gives the Pad\'e approximation to $(1-z)^{1/r}$ by 
hypergeometric polynomials and some properties of the 
approximating polynomials.
	\begin{lemma}\label{hyp}
	Let $n$ be a positive integer and $g \in \{ 0 , 1 \}$. Put
	\begin{eqnarray}\label{AB}\nonumber
	A_{n, g} (z) & = &  \sum_{m =0}^{n}{n - g + \frac{1}{r} \choose m} {2n - g - m \choose n - g}   (-z)^{m},   \\  
	B_{n, g} (z) & = & \sum_{m =0}^{n-g}{n - \frac{1}{r} \choose m} {2n - g - m \choose n }   (-z)^{m}.
	 \end{eqnarray}
	  \flushleft
	 \begin{itemize}
	 \item[(i)] There exists a power series $F_{n,g}(z)$ such that for all complex numbers $z$ with $|z| < 1$
	 \begin{equation}\label{ABF}
	 A_{n,g}(z) - (1 - z)^{1/r} B_{n, g}(z) = z^{2n+1 -g}F_{n,g}(z)
	 \end{equation}
	  and 
	  \begin{equation}\label{F}
	 |F_{n,g}(z)| \leq \frac{{n-g+1/r \choose n+1-g} {n- 1/r \choose n}}{{2n + 1 - g \choose n}} (1 - |z|)^{-\frac{1}{2}(2n + 1 - g)}.
	 \end{equation}
	 \item[(ii)] For all complex numbers $z$ with $|1 - z| \leq 1$ we have 
	 \[
	 |A_{n,g}(z)| \leq {2n - g \choose n}.
	 \]
	Further, if $|1 - z| \leq2$, then
	 \begin{equation}\label{A}
	 |A_{n,g}(z)| \leq 2^{3n+2}.
	 \end{equation}
 \item[(iii)] For all complex numbers $z \neq 0$ and for $I \in \{0 , 1\}$ we have
	 \begin{equation}\label{BA}
	 A_{n, 0}(z) B_{n+I , 1}(z) \neq A_{n+I , 1}(z) B_{n , 0}(z).
	 \end{equation} 
         \item[(iv)] For all $\lambda \in \mathcal{O}$ and 
         $c \in \mathbb{Z}$ we have
$$
A^{*}_{n,g}(\lambda, r^2 \sqrt{D} c) \in \mathcal{O}
$$
and
  $$
  B^{*}_{n,g}(\lambda, r^2 \sqrt{D} c) \in \mathcal{O}.
  $$
	 \end{itemize}
	\end{lemma}
\begin{proof}
We first prove (ii).
Put
	$$
	C_{n,g}(z) = \sum_{m =0}^{n} {n - 1/r \choose n - m }{n- g + 1/r \choose m} z^{m} 
	$$
	and 
	$$
	D_{n,g} = \sum_{m = 0}^{n - g} {n - 1/r \choose m} {n - g + 1/r \choose n - g - m} z^{m}.
	$$
Note that, in terms of hypergeometric functions,
	$$
	A_{n, g}(z) = {2n - g \choose n} F (-1/r - n+ g , -n , -2n + g , z),
	$$ 
	$$
	B_{n, g}(z) = {2n - g \choose n - g} F (1/r - n  , -n+ g , -2n + g , z), 
	$$ 
	$$
	C_{n, g}(z) = {n - 1/r \choose n} F (-1/r - n + g , -n, 1- \frac{1}{r}, z)
	$$
and 
	$$
	D_{n, g}(z) = {n- g + 1/r \choose n- g} F (1/r- n  ,  -n + g , 1+ \frac{1}{r},z).
	$$ 
We show below that
	$$
	C_{n,g}(z) = A_{n,g}(1 - z)  , \  D_{n , g}(z) = B_{n,g}(1 - z).
	$$
The power series $ F(z) = \sum_{m =0}^{\infty} a_{m} z^{m}$ is a solution to the differential equation (\ref{hip}) 
precisely when
	\begin{equation}\label{an}
	(k + 1) (\gamma + k) a_{k + 1}  = (\alpha + k) ( \beta + k ) a_{k}   \  \textrm{for}  \  k = 0, 1 , 2 , \ldots  .
	\end{equation}
Therefore if $\gamma>0$, all the coefficients of $F(z)$ are determined by $a_0$. Hence the solution space of (\ref{hip}) is one-dimensional. Both $A_{n,g}(1 - z) $ and $C_{n,g}(z)$ satisfy  (\ref{hip}) with $\alpha = -1/r - n + g$, $\beta = -n$, $\gamma = 1 -\frac{1}{r}$. Hence they are linearly dependent. On equating the coefficients of $z^{n}$ in 
	$$
	(1 + z) ^{2n - g} = (1 + z)^{n - 1/r} (1 + z)^{n - g + 1/r}, 
	$$
we find that
	\begin{displaymath}\nonumber
	C_{n,g} (1) =  \sum_{m = 0}^{n} {n - 1/r \choose n -m }{n - g + 1/r \choose m} = {2n - g \choose n} = A_{n,g}(0),
	\end{displaymath}
and hence $C_{n,g}(z) = A_{n,g}(1 - z)$. Similarly, $D_{n,g}(z) = B_{n,g}(1 - z)$.
One can easily observe that $C_{n, g}(z)$  has positive coefficients. Hence when $|1 - z| \leq 1$,
	$$
	|A_{n,g}(z)| = |C_{n,g}(1 - z)| 
\leq C_{n,g}(1) = A_{n,g}(0) = {2n - g \choose n}.
	$$
Similarly, if $|1 - z| \leq 2$, we have
	$$
	|A_{n,g}(z)| = |C_{n,g}(1 - z)| 
\leq C_{n,g}(2)\leq 2^{3n+2}. 
	$$
This proves part (ii) of our lemma.

Next we prove (\ref{ABF}).  Define
	  $$
	  G_{n,g}(z) = F(n+1-g , n+1 - \frac{1}{r} , 2n +2 -g , z)
	  $$ 
 and notice that, for $|z| < 1$, the functions $A_{n,g}(z)$, $(1 - z)^{1/r} B_{n,g}(z)$ and $z^{2n+1 -g}G_{n,g}(z)$ satisfy (\ref{hip}) with $\alpha = -1/r - n + g$ , $\beta = -n$, $\gamma = -2n + g$.
 Suppose
	  \begin{displaymath} 
	  G_{n,g}(z) = \sum_{m =0}^{\infty} g_{m} z^{m}.
	  \end{displaymath}
We have $g_{0} = 1$ and, for $m \geq 0$,
	 $$
	  \frac{g_{m+1}}{g_{m}}  =   \frac{(n + 1- g + m)(n + 1-\frac{1}{r} +m)}{(m + 1)(2n+ 2 - g + m)}$$
	  $$\leq  \frac{n + 1/2 - g/2 + m}{m + 1} = \frac {(-1)^{m+1} {-n -1/2 + g/2 \choose m+1 }}{(-1)^{m}{-n - 1/2 + g/2 \choose m}}.
	$$
Therefore,
	  $$
	  |G_{n,g}(z)| \leq \sum_{m = 0}^{\infty} {-n -1/2 + g/2 \choose m} (-|z|)^{m} = (1 - |z|)^{-\frac{1}{2}(2n+1-g)}.
	  $$
Since $n \geq 1$ and $g \in \{0 , 1 \}$, $\gamma = -2n + g$ is a negative integer. By (\ref{an}), if $F(z) = \sum_{m= 0}^{\infty} a_{m} z^{m}$ is a solution to (\ref{hip}), then since $a_{0}$ and $a_{2n - g+1}$  may vary independently, the solution space of (\ref{hip}) is two-dimensional. Therefore, there are constants $c_{1}$, $c_{2}$ and $c_{3}$, not all zero, such that 
	 $$
	 c_{1}A_{n,g}(z) + c_{2} (1 - z) ^{1/r}B_{n,g}(z) + c_{3} z^{2n + 1 - g} G_{n,g}(z) = 0.
	 $$
 Letting $z = 0$, since $A_{n,g}(0) = B_{n,g}(0) \neq 0 $, we find that $c_{1} = - c_{2} \neq 0$. We may thus assume $c_{1} = 1$. Substituting $z=1$ in the above identity we get 
 $c_{3} = -\frac{A_{n,g}(1)}{G_{n,g}(1)}$, whence we may take
	 $$
	 F_{n,g}(z) = A_{n,g}(1) G_{n,g}(1)^{-1} G_{n,g}(z)
	 $$ 
to obtain \eqref{ABF}.
 In order to complete the proof of part (i), note that,  by (\ref{Gam}), we have
	 \begin{eqnarray*}
	& &  A_{n,g}(1) G_{n,g}(1)^{-1}  = {n - 1/r \choose n} \frac{\Gamma(n +1) \Gamma(n + 1 + \frac{1}{r} -g)}{\Gamma(2n + 2 -g) \Gamma(1/r)}\\
	 & =& \frac{ {n - 1/r \choose n} {n - g + 1/r \choose n + 1 - g}}{{2n+1-g \choose n}}.
	\end{eqnarray*}
Hence \eqref{F} holds.
    Now we prove (iii). By (\ref{ABF}),
   $$
   A_{n,0}(z) B_{n+ I , 1}(z) - A_{n+I , 1}(z)B_{n,0}(z) = z^{2n+I}P_{n,I}(z),
   $$
   where $P_{n,I}(z)$ is a power series. However, the left hand side of the above identity is a polynomial of degree at most $2n + I$, and so $P_{n,I}$ must be a constant.  Letting $z = 1$, we obtain that 
   	\[
   	P_{n,I}(1)={n - 1/r \choose n}{n +I+1/r-1 \choose n+I-1}-{n +I-1/r \choose n+I}{n + 1/r \choose n}\neq 0.
   	\]
Therefore, $$A_{n,0}(z) B_{n+I , 1}(z) - A_{n+I , 1}(z)B_{n,0}(z) = 0$$ if 
and only if $z = 0$.

Finally we prove (iv). From the definition of $A_{n,g}$ and $B_{n,g}$
it suffices to prove that
$$
{a/r \choose m} (r^2 \sqrt{D})^m \in \mathcal{O}
$$
for every $a , m \in \mathbb{Z}$ with $m \geq 0$. This  clearly holds
for $m =0$. Let $m \geq 1$. The assertion follows if we show that for
each prime $p$ the number
$$
t(m): = {a/r \choose m} (r)^{2m}=a(a-r)\ldots (a-r(m-1))r^m/m!
$$
 is a $p$-adic  integer. Observe that  
	\begin{equation}\label{order}
	\textrm{ord}_p(m!)=\sum_{j=1}^{\infty}\left[ \frac{m}{p^j}\right] < \frac{m}{p-1} \leq m
	\end{equation}
where $\textrm{ord}_p(n)$ denotes the exponent with which $p$ divides an
integer $n.$
If $p | r$ then we notice that $t(m)$ is a $p$-adic integer.
If $p\nmid r,$ from the solvability of the congruence 
	\[
	a-rx \equiv 0 \pmod{p^j},
	\]
we get
	$$
	\textrm{ord}_p(a(a-r)\ldots (a-r(m-1)) \geq \sum_{j=1}^{\infty}\left[
	  \frac{m}{p^j}\right]
	$$
and the assertion follows from \eqref{order}.
\end{proof}
\section{Construction of some Algebraic Numbers}\label{AN}
Let $F(x , y)$ be a diagonalizable form given in \eqref{form} and $D$ be the discriminant of the associated  quadratic form $Ax ^2 + B xy + Cy^2$ given in \eqref{chiquad}. We will work with the quadratic number field $\mathbb{Q}(\sqrt{D})$ and its ring of integers $\mathcal{O}$. 
We use the approximating polynomials $A_{n,g}(z)$ and $B_{n,g}(z)$ in 
Section \ref{Pade} to construct some complex sequences.
Suppose that $(x_1,y_1)$ and $(x_2,y_2)$ are two distinct solutions to the
inequality $0 < |F(x , y)| \leq h$ that are related to a fixed root of unity. Let $Z$ be the function defined in \eqref{defofZ}. We define
$$
(X_1,Y_1)=\begin{cases}
     (u_1,v_1)\ {\rm if}\ Z_1=|u_1|\\
     (v_1,u_1)\ {\rm otherwise}
    \end{cases}
$$
and
$$
(X_2,Y_2)=\begin{cases}
     (u_2,v_2)\ {\rm if}\ (X_1,Y_1)=(u_1,v_1)\\
     (v_2,u_2)\ {\rm otherwise}.
    \end{cases}
$$

If $X_1Y_1$, $X_2Y_2\neq0$, define 
$$
	 \Sigma_{n,g} = \frac{Y_{2}}{X_{2}}A_{n,g}(z_{1}) - 
\frac{Y_{1}}{X_{1}} B_{n,g}(z_{1})
	 $$
with  $z_{1} = 1 - Y_{1}^r/ X_{1}^r$ and 
$$
 \tilde{\Sigma}_{n,g} = \frac{X_{2}}{Y_{2}}A_{n,g}(\tilde{z_{1}}) -
\frac{X_{1}}{Y_{1}} B_{n,g}(\tilde{z_{1}})
 $$
with $\tilde{z_{1}} = 1-X_{1}^r/Y_{1}^r$.
Further let
\begin{equation}\label{cng}
c_{n,g}=
 r^n \left(r (r-1)
  \sqrt{D}\right)^{n+g}\left(\frac{2}{\chi}\right)^{1-g},
\end{equation}
	$$
	\Lambda_{n,g} = c_{n,g} 
X_{1}^{rn+1-g}X_{2} \Sigma_{n,g}\ \textrm{ and }\ \tilde{\Lambda}_{n,g} = c_{n,g}
Y_{1}^{rn+1-g}Y_{2} \tilde{\Sigma}_{n,g}.
	$$
We will show that  $\Lambda_{n,g}$ is either an integer in 
$\mathbb{Q}(\sqrt{D})$ or an $r$-th root of such an integer. The same assertion holds for $\tilde\Lambda_{n,g}.$  If $\Sigma_{n,g} \neq 0,$ we can get a lower bound for
 $|\Lambda_{n,g}\tilde\Lambda_{n,g}|$. We shall use Lemma \ref{hyp} to get 
an upper bound also.
As a direct consequence of Lemmas \ref{Siegelshows} and \ref{hyp} (iv), 
we get the following result.
	\begin{lemma}\label{24}
	For any pair of integers $(x , y)$ satisfying 
$0<|F(x,y)|\leq h$, we have 
	$$A^{*}_{n,g}( r^2 (r-1) \sqrt{D}X_1^{r},  
r^2 (r-1) \sqrt{D}X_1^{r} - r^2 (r-1) \sqrt{D}Y_1^{r})$$
	  and 
	  $$B^{*}_{n,g}( r^2 (r-1) \sqrt{D}X_1^{r},  r^2 (r-1) \sqrt{D} X_1^{r}
-  r^2 (r-1) \sqrt{D}Y_1^{r})$$
	   are in 
	$\mathcal{O}$.
	\end{lemma}
 	\begin{lemma}\label{lb}
 	If $\Sigma_{n,g}\neq 0$, then we have
\[
|\Lambda_{n,g}\tilde\Lambda_{n,g}|\geq 1.
\]
   	\end{lemma}
\noindent {\bf Note.} When $D<0,$ we have $|\Lambda_{n,g}|=|\tilde\Lambda_{n,g}|$
as $\Lambda_{n,g}$ and $\tilde\Lambda_{n,g}$ are complex conjugates.
 	\begin{proof}
Note that $A_{n,g}$ is of degree $n$ and $B_{n,g}$ is of degree $n-g.$
We re-write
	 \begin{eqnarray*}
	&& \Lambda_{n, g} =\\
	&&  \left(r (r-1) \sqrt{D}\right)^{g}
         \left(\frac{2}{\chi} \right)^{1-g} \left [X_{1}^{1 -g} Y_{2}
           \mathcal{A}^{*}_{n,g}-  
X_{1}^{(r-1)g}X_{2} Y_{1}(r^2(r-1)\sqrt{D})^g\mathcal{B}^{*}_{n,g}\right],
	 \end{eqnarray*}
 where
	$$
	\mathcal{A}^{*}_{n,g}  = A^{*}_{n,g}(r^2 (r-1) \sqrt{D} X_1^r , 
r^2 (r-1) \sqrt{D}(X_1^r - Y_1^r))
	$$
and
$$    
	\mathcal{B}^{*}_{n,g} = B^{*}_{n,g}(r^2 (r-1) \sqrt{D}X_1^r ,  
r^2 (r-1) \sqrt{D}(X_1^r - Y_1^r)).
	$$
  Let $g = 0.$ It follows by  Lemmas \ref{ai2Diag} and \ref{24} that
	  $$
	 \Lambda_{n, 0} = \frac{2}{\chi} X_{1} Y_{2} \mathcal{A}^{*}_{n,0} -   \frac{2}{\chi} X_{2} Y_{1}\mathcal{B}^{*}_{n,0}
	 $$
is in  $\mathcal{O}$.\\
 Let $g = 1$.  Then	 
\begin{eqnarray*}
	&& \Lambda_{n, 1}^r = 
(r(r-1)\sqrt{D})^r\left(Y_2A_{n,1}^*-
(r^2(r-1)\sqrt{D})X_1^{r-1}X_2Y_1B_{n,1}^*\right)^r\\
&=& (r(r-1)\sqrt{D})^r\sum_{a=0}^r (-1)^{r-a}{r\choose a}
Y_2^a Y_1^{r-a} \left({r\choose 1} X_1^{r-1}X_2\right)^{r-a} \times\\
&&(A_{n,1}^*)^a(r(r-1)\sqrt{D})^{r-a}(B_{n,1}^*)^{r-a}.
\end{eqnarray*}
Hence by Lemmas \ref{ai2Diag} and \ref{24}, we get $\Lambda_{n,1}^r \in \mathcal{O}$. \\
 As above one can prove that 
$\tilde\Lambda_{n,0}$,  
$\tilde\Lambda^{r}_{n,1}\in \mathcal{O}.$ 
Further since $X_1^r-Y_1^r\in \mathbb Z,$ one can easily see that 
$\Lambda_{n,0}$, $\tilde{\Lambda}_{n,0}$ and $\Lambda_{n,1}^r$, $\pm\tilde{\Lambda}_{n,1}^r$ are algebraic
conjugates in $\mathcal{O}.$ Hence we get
 $$
  |\Lambda_{n , g}\tilde\Lambda_{n,g}|\geq 1.
  $$
\end{proof}

	\begin{lemma}\label{c}
Let $F$ be a diagonalizable form given by \eqref{form}. 
Let $(x_1,y_1)$ and $(x_2,y_2)$ be two solutions of \eqref{firstineq} 
related to a fixed $r$-th root of unity, say $\omega$,
with
$\zeta_2\leq \zeta_1.$
Assume that $Z_1^r>2h$ and $\Sigma_{n,g} \neq 0.$ Suppose 
	$$
\Lambda'_{n,g}=c_{1}(n,g)  \,h Z_1^{nr+1-g}Z_{2}^{-r+1}+ c_{2}(n,g) \, 
h^ {2n+1-g}  Z_{1}^{-r(n+1-g)+1-g}Z_{2}
	$$ 
	where 
	$$
	c_{1}(n,g) = 2^{3n+2}|c_{n,g}|
	$$
	and
	$$
	c_{2}(n,g) =  2^{n+1-g}|c_{n,g}|
\left(1-\frac{2h}{Z_1^r}\right)^{-\frac{1}{2}(2n+1-g)}
\frac{|{n-g+1/r \choose n+1-g} {n- 1/r \choose n}|}{{2n + 1 - g \choose n}}
	$$
with $c_{n,g}$ given by \eqref{cng}. 
Then
$$\Lambda'_{n,g} \geq 1.$$
	\end{lemma}
 \begin{proof}
Since $Z_1^r>2h$, $\zeta_1<1/2$. Let $i\in\{1,2\}$.
If $D<0$, then $|X_i|=|Y_i|$. Hence we may assume that $X_iY_i\neq 0$. If $D>0$, as $\zeta_1<1/2$, $X_1^r$ and $Y_1^r$ are not zero and are of the same sign and $|X_1|<2^{1/r}|Y_1|$. Thus in either case $z_1=1-\frac{Y_1^r}{X_1^r}$ and $\tilde z_1=1-\frac{X_1^r}{Y_1^r}$
satisfy $|z_1|<1,|1-z_1|\leq 1,$
$$|\tilde z_1|=|F/Y_1^r|<1.$$
Also
$$0<\left|\frac{X_1^r}{Y_1^r}\right|=|1-\tilde z_1|< 2.$$ 
First suppose that $(X_1,Y_1)=(u_1,v_1)$. 
Since we assumed $(x_1,y_1)$ is related to $\omega,$ we have 
 $$
 (1-z_1)^{1/r}=Y_1\omega/X_1.
 $$ 
Then by \eqref{ABF}, we get
$$
|\Lambda_{n,g}| \leq
|c_{n,g}X_1^{rn+1-g}X_2| \left|  \left(\frac{\omega Y_{2}}{X_{2}} - 1
  \right)
A_{n,g}(z_{1})+ z_{1}^{2n+1-g}F_{n,g}(z_{1})   \right|.
	$$
Note that 
$$|z_1| = \frac{|F|}{|X_1|^r}\leq \frac{h}{|X_1|^r}.$$
Hence 
$$1-|z_1| \geq 1-\frac{h}{|X_1|^r}.$$
We apply the above inequality together with the estimates from
\eqref{Gap22}, \eqref{Gap23}, \eqref{F}  and \eqref{A}
to get 
\begin{eqnarray}\label{lambda}
&& \qquad \qquad   |\Lambda_{n,g}| \leq
|c_{n,g}|\ h\ Z_1^{rn+1-g}Z_2^{-r+1}2^{3n+2}+ \\ \nonumber
&+&|c_{n,g}|\left(h^ {2n+1-g}  Z_{1}^{-r(n+1-g)+1-g}Z_{2}
\left(1-\frac{h}{Z_1^r}\right)^{-\frac{1}{2}(2n+1-g)}\mathfrak{C}\right),
\end{eqnarray}
where $\mathfrak{C} = \frac{|{n-g+1/r \choose n+1-g} {n- 1/r \choose n}|}
{{2n + 1 - g \choose n}}$.
The above estimate also holds when $(X_1,Y_1)$ $=(v_1,u_1)$.
Similarly we have
\begin{eqnarray*}
&&|\tilde\Lambda_{n,g}| \leq
|c_{n,g}|\ h\ |Y_1|^{rn+1-g}Z_2^{-r+1}2^{3n+2}+ \\ 
&+&|c_{n,g}|\left(h^ {2n+1-g}  |Y_{1}|^{-r(n+1-g)+1-g}|Y_{2}|
\left(1-\frac{h}{|Y_1|^r}\right)^{-\frac{1}{2}(2n+1-g)}\mathfrak{C}\right).
\end{eqnarray*}
Using $|Y_1|\leq |X_1| <2^{1/r}|Y_1|,$ we get
\begin{eqnarray}\label{tildelambda}
& & \qquad \qquad  \qquad \qquad \qquad |\tilde\Lambda_{n,g}| \leq
|c_{n,g}|\ h\ Z_1^{rn+1-g}Z_2^{-r+1}2^{3n+2}+ \\ \nonumber
&&|c_{n,g}|\left(h^ {2n+1-g}2^{n+1-g-(1-g)/r}  Z_{1}^{-r(n+1-g)+1-g}Z_{2}
\left(1-\frac{2h}{Z_1^r}\right)^{-\frac{1}{2}(2n+1-g)}
\mathfrak{C}\right).
\end{eqnarray}
The above bound is valid for $|\Lambda_{n,g}|$ as well. 
Now we combine these upper bounds with the lower bound for $|\Lambda_{n,g}
\tilde\Lambda_{n,g}|$ in Lemma \ref{lb} to get assertion of the lemma.
\end{proof}

The bounds that were obtained for the size of algebraic numbers  above are useful when those numbers are non-zero.  Following an argument of Bennett in \cite{Ben2}, we show that $\Sigma_{n , g}$'s do not vanish often.

\begin{lemma}\label{nv}
 If $n \in \mathbb{N}$ and $I \in \{ 0, 1 \}$, then at most one of 
$$\left\{ \Sigma_{n,0}, \Sigma_{n+I,1}\right\}
$$ can vanish.
  \end{lemma}
\begin{proof}
    Let $n$ be a positive integer and  $I \in \{ 0 , 1\}$ . Following an argument of Bennett \cite{Ben2}, we define the matrix $\mathbf{M}$:
   \begin{displaymath}
  \mathbf{M} =
  \left( \begin{array}{ccc}
  A_{n,0}(z_{1}) & A_{n+I , 1}(z_{1}) &Y_1/X_1 \\
  A_{n,0}(z_{1}) & A_{n+I , 1}(z_{1}) &Y_1/X_1 \\
 B_{n,0}(z_{1}) & B_{n+I , 1}(z_{1}) & Y_2/X_2 
 \end{array} \right) .
 \end{displaymath}
 The determinant of $\mathbf{M}$ is zero because it has two identical rows. Expanding along the first row, we find that
  \begin{eqnarray*}
 && A_{n, 0}(z_{1}) \Sigma_{n+I , 1} -   A_{n+I, 1}(z_{1}) \Sigma_{n, 0} +\\
&&  + \frac{Y_1}{X_1}( A_{n, 0}(z_{1}) B_{n+I ,1}(z_{1})  - A_{n+I, 1}(z_{1}) B_{n ,0}(z_{1}) )
  \end{eqnarray*}
  vanishes and hence if $\Sigma _{n,0} = \Sigma _{n + I,1} = 0$,  then 
  $$
  A_{n, 0}(z_{1}) B_{n+I ,1}(z_{1})  - A_{n+I, 1}(z_{1}) B_{n ,0}(z_{1}) = 0,  
  $$
  contradicting part (iii) of Lemma \ref{hyp}.
 \end{proof}
\section{An  Auxiliary Lemma }\label{AAL}

 We devote this section to a lemma that will be used later in combination with our gap principle to complete the proof of our main theorems. 
 We will  repeatedly appeal to the induction procedure that is introduced  in the proof of  Lemma \ref{induction} in proving  our main theorems  in later sections.
 The ideas used in the proof are originally  due to Evertse in
 \cite{Eve1}, where he proved upper bounds for the number of solutions
 to cubic Thue equations.

The computational steps in the proof are
 verified by symbolic computation in MATHEMATICA. 
A file containing
 these steps is available at
  http://www.math.tifr.res.in/$\sim$saradha/hypergeometric.txt.

\begin{lemma}\label{induction}
Let $F(x, y)$ be a diagonalizable form of degree  $r \geq 5.$ Let $\omega$ be a fixed $r$-th root of unity, $S_\omega$ be the set of all primitive solutions to $0<|F(x , y)| \leq h$ that are related to $\omega$ and $S'_{\omega}$ be as in \eqref{defofS'}. Assume  $ |S'_{\omega}| = k \geq 3.$
Further, suppose that 
	\begin{equation}\label{j}
	|j|\geq 2 r^{i_{7}/r} h^{i_{8}/r}
	\end{equation}
with
	\be\label{ji1i2}
	i_{7}=\frac{7r^2}{R(k)-2r-1}, \, \, 
	i_{8}=\frac{2R(k)+r^2-3r}{R(k)-2r-1}.
	\ee
Then for every integer $n\geq 1,$ we have
	\be\label{ind_star}
Z_k\geq
\frac{Z_{k-1}^{(n+1)r-1}}{2^{n+4} r^{(3nr+2)/(r-2)}|j|^{(nr+2)/(r-2)}h^{2n+1}}.
	\ee
\end{lemma}
 \begin{proof}
Let $(x_{1}, y_{1}), \ldots, (x_{k}, y_{k})$ be the elements of  $S'_{\omega}$ indexed such that 
$\zeta_{i+1}\leq \zeta_{i}$ for $i=1, \ldots,  k-1$.
 By \eqref{j}, $$|j|\geq 2^{1+(r-2)/(r(R(k-1)-1))}h^{2/r}.$$
By the Remark following Definition 5.3, we see that $\zeta_i<1$ for $1\leq i \leq k.$
From Lemma \ref{43}, we get
\begin{equation}\label{u11}
Z_i\geq \frac{|j|}{2 h^{1/r}} \textrm{ for } i=1,\ldots,k.
\end{equation}
By Lemma \ref{gap_Siegel} we have
\begin{equation}\label{uk}
Z_{i}\geq \frac{|j|}{2h}Z_{i-1}^{r-1},
\end{equation}
for $i =2, \ldots, k$.
Therefore we obtain
\begin{equation}\label{uk1}
Z_k\geq \left(\frac{|j|}{2h}\right)^{\frac{R(k)-1}{r-2}}Z_1^{R(k)},
\end{equation}
where $R(k)$ is as in Definition \ref{defofR(k)}.
At many instances later, we will use the above inequality with $k$ replaced by 
$k-1.$
Thus we derive that $Z_{k-1}^r>2h$ and hence $\zeta_{k-1}<1/2$.
Now, we use Lemma \ref{c} with $(x_1,y_1)$ and $(x_2,y_2)$ replaced by
$(x_{k-1},y_{k-1})$ and $(x_{k},y_{k})$, respectively. For $n\in\mathbb{N}$,
we define
	$$\Lambda'_{n,g}:=\Xi_{n,g}+\Pi_{n,g},$$
where
	$$\Xi_{n,g}=c_{1}(n,g)  \,h Z_{k-1}^{nr+1-g}Z_{k}^{-r+1}$$
and
	$$\Pi_{n,g}=c_{2}(n,g) \, 
	h^ {2n+1-g}  Z_{k-1}^{-r(n+1-g)+1-g}Z_{k}.
	$$

In order to give upper bounds for the quantities
$c_1(n,g)$ and $c_2(n,g)$, we 
recall that by \eqref{LGauss} we have
$$|\chi| \geq \frac{1}{r^{4/(r-2)}|j|^{2/(r-2)}}$$
and by definition $j^2=\chi^2 D$.
Then
\begin{equation}\label{c1}
|c_1(n,g)| \leq 2^{3n+3}r^{\frac{r(2g+3n)+2}{r-2}}
|j|^{\frac{r(g+n)+2}{r-2}}.
\end{equation}
By \eqref{u11} and \eqref{uk1}, we get
$$\left(1-\frac{2h}{Z_{k-1}^r}\right)^{-\frac{1}{2}(2n+1-g)}\leq 2^{n+1}.$$
Since ${2n+1-g \choose n}\geq 2^{2n}/(n+1)$ and 
$|{n-g+1/r \choose n+1-g} {n- 1/r \choose n}|\leq 1,$ we get
\begin{equation}\label{c2}
|c_2(n,g)| \leq 2^{n+3}r^{\frac{r(2g+3n)+2}{r-2}}
|j|^{\frac{r(g+n)+2}{r-2}}.
\end{equation}
By \eqref{c1},  
$|\Xi_{n,g}|\leq 1/2$ if
\begin{equation}\label{Xi}
2^{3n+4}r^{\frac{r(2g+3n)+2}{r-2}}
|j|^{\frac{r(g+n)+2}{r-2}}h Z_{k-1}^{nr+1-g}\leq Z_k^{r-1}.
\end{equation}
 Then if $\Sigma_{n,g}\neq 0$, by Lemma \ref{c} and \eqref{c2}, we get
	\begin{equation}\label{u31}
	Z_k\geq \frac{Z_{k-1}^{r(n+1-g)-1+g}}{2^{n+4}r^{(r(2g+3n)+2)/(r-2)}|j|^{(r(g+n)+2)/(r-2)}
	h^{2n+1-g}}.
	\end{equation}
For a given integer $n \geq 1$ and $g \in \{0,1\}$ let $a_i=a_i(n,g,r), 1\leq i\leq 5$ be some rational numbers which will be specified at different stages. 
We say that {\it property $P[a_1,a_2,a_3,a_4,a_5]$} holds if
$$Z_k\geq \frac{Z_{k-1}^{a_1}}{2^{a_2}r^{a_3}|j|^{a_4}h^{a_5}}.$$
\vskip 2mm
\noindent
Suppose $P[a_1,a_2,a_3,a_4,a_5]$ holds and $a_2+a_4\geq 0$. Then the inequality  \eqref{Xi} holds if
\begin{eqnarray}\label{k-1*}
&&Z_{k-1}^{a_1(r-1)-nr-1+g}\geq \\ \nonumber
&&2^{a_2(r-1)+3n+4}r^{a_3(r-1)+\frac{r(2g+3n)+2}{r-2}}
|j|^{a_4(r-1)+\frac{r(g+n)+2}{r-2}}h^{a_5(r-1)+1}.
\end{eqnarray}
Suppose $A_1:=a_1(r-1)-nr-1+g>0$. By \eqref{uk1}, the inequality \eqref{k-1*} holds if
\begin{eqnarray}\label{u1}
& & Z_1^{A_1R(k-1)}\ |j|^{A_1\left(\frac{R(k-1)-1}{r-2}\right)-a_4(r-1)-
\frac{r(g+n)+2}{r-2}}
\geq\\ \nonumber 
& & 2^{A_1\left(\frac{R(k-1)-1}{r-2}\right)+
    a_2(r-1)+3n+4}r^{a_3(r-1)+\frac{r(2g+3n)+2}{r-2}}
h^{A_1\left(\frac{R(k-1)-1}{r-2}\right)+a_5(r-1)+1}.
\end{eqnarray}
By \eqref{u11}, the inequality 
\eqref{u1} holds if
\begin{equation}\label{j2rh}
|j|^{B_1}\geq 2^{B_2}r^{B_3}h^{B_4},
\end{equation}
where
$$B_1=A_1
\left(\frac{R(k)-1}{r-2}\right)-a_4(r-1)-\frac{r(g+n)+2}{(r-2)},$$
$$B_2= A_1
\left(\frac{R(k)-1}{r-2}\right)+a_2(r-1)+3n+4,$$
$$B_3=a_3(r-1)+\frac{r(2g+3n)+2}{r-2}$$
and 
$$B_4=A_1\left(\frac{2R(k)-r}{r(r-2)}\right)+a_5(r-1)+1.$$
Using \eqref{j}, if the following conditions $(i)$--$(iv)$ hold, then \eqref{j2rh} and hence \eqref{Xi} are valid.
$$(i)\ A_1>0,$$
$$(ii)\ B_1>0,$$
$$(iii)\ B_1\times i_{7}\geq r(B_3+(B_2-B_1)/2)$$
$$(iv)\ B_1\times i_{8}\geq r B_4.$$
Then if $\Sigma_{n,g}\neq 0$, \eqref{u31} holds.

To complete the proof we use  induction on $n$. First we verify the basis of induction by implementing MATHEMATICA.
By \eqref{uk}, $P[r-1,1,0,-1,1]$ holds.
Fix $(n,g)=(1,0)$. Suppose $\Sigma_{1,0}\neq 0$. 
We check that $(i)-(iv)$ are valid with $a_1=r-1$, $a_2=1$, $a_3=0$, $a_4=-1$ and $a_5=1$. 
Hence by \eqref{u31}, $P[2r-1,5,(3r+2)/(r-2),(r+2)/(r-2),3]$ holds, i.e. \eqref{ind_star} is 
valid with $n=1$.

If
$\Sigma_{1,0}= 0$, then by Lemma \ref{nv}, both $\Sigma_{1,1}$ and 
$\Sigma_{2,1}$ are non-zero. Fix $(n,g)=(1,1)$. We verify that $(i)-(iv)$ 
are valid with $a_1=r-1$, $a_2=1$, $a_3=0$, $a_4=-1$ and $a_5=1$. Hence $P[r,5,(5r+2)/(r-2),(2r+2)/(r-2),2]$ holds.
 Now we 
fix $(n,g)=(2,1)$ and check that $(i)-(iv)$ are valid with 
$a_1=r$, $a_2=5$, $a_3=(5r+2)/(r-2)$, $a_4=(2r+2)/(r-2)$ and $a_5=2$ since 
$r \geq 5$ and $k \geq 3.$
Thus $P[2r,6,(8r+2)/(r-2),(3r+2)/(r-2),4]$ holds. Hence \eqref{ind_star} is valid for $n=1$ provided
	\be\label{conc}
	\frac{Z_{k-1}}{2r^{5r/(r-2)}|j|^{2r/(r-2)}h}\geq 1.
	\ee
By \eqref{uk1} with $k$ replaced by $k-1$ and \eqref{u11} with $i=1$, this is valid if
\begin{equation}\label{final}
|j|\geq 2^{(R(k)+r-3)/(R(k)-2r-1)}r^{5r/(R(k)-2r-1)}h^{(2R(k)+r^2-3r)/(r(R(k)-2r-1))}.
\end{equation}
The above inequality holds by \eqref{j}.

Now we proceed by assuming  that $P[(n+1)r-1,n+4,(3nr+2)/(r-2),(nr+2)/(r-2),2n+1]$ holds for
some $n\geq 1$.
We will  show that the property holds for $n+1$. Fix $(n,g)=(n+1,0).$
Suppose $\Sigma_{n+1,0}\neq 0.$   
We verify that $(i)-(iv)$ are valid. Hence
 the property follows for $n+1$.\\

Next we suppose that $\Sigma_{n+1,0}=0.$ Then by Lemma \ref{nv} both $\Sigma_{n+1,1}$ and $\Sigma_{n+2,1}$ are non-zero. First fix $(n,g)=(n+1,1)$ and check that $(i)-(iv)$ are valid so that 
$P[(n+1)r,n+5,((3n+5)r+2)/(r-2),(r(n+2)+2)/(r-2),2n+2]$ holds. \\

Now fix $(n,g)=(n+2,1)$. Proceeding as above, we obtain
 that $P[(n+2)r,n+6,((3n+8)r+2)/(r-2),((n+3)r+2)/(r-2),2n+4]$ holds.
Thus property $P$ holds for $n+1$ by \eqref{conc}.
This completes the induction.
\end{proof}
\section{Proof of Theorem \ref{thml}}\label{proofofthml}
Let  $F(x , y)$  be a diagonalizable form given in \eqref{form}. If $F(x , y)$ is a definite form, \eqref{deltainrj} and
Corollary \ref{definite} imply Theorem \ref{thml}. Let $\omega$ be an $r$-th root of unity and  $S_\omega$ be the set of all solutions of the inequality  $0 < |F(x , y)| \leq h$ that are related to $\omega$. Let $k = |S'_{\omega}|$ and $m\geq 3$ be the integer in the statement of Theorem \ref{thml} that satisfies \eqref{Delta_condn} and \eqref{twoalphas}. We will show that  $k \leq m-1$. Assume $k \geq m$. Then
by \eqref{Delta_condn}, we have
\begin{equation}
	\Delta' \geq  r^{\mathcal{A}_1}h^{\mathcal{A}_2}
	\end{equation}
where 
	\begin{equation}
	\mathcal{A}_1=\frac{7r^2(r-1)}{(r-1)^{k-1}-2r-1}\ \textrm{ and }\
	\mathcal{A}_2=\frac{(r-1)(r^2+r+2)}{(r-1)^{k-1}-2r-1}.
	\end{equation}
Therefore, by \eqref{deltainrj}, we find that the inequality \eqref{j} holds.  
By Lemma \ref{induction}, \eqref{u11} and \eqref{uk1}, for every integer $n \geq 1$, we have
\begin{equation}\label{TAE}
Z_k> \frac{|j|^{C_1}}{2^{C_2}r^{C_3}h^{C_4}},
\end{equation}
with 
$$C_1=(((n+1)r-1)(R(k)-1)-(nr+2))/(r-2),$$
$$C_2=(((n+1)r-1)(R(k)-1)+(n+4)(r-2))/(r-2),$$
$$C_3=(3nr+2)/(r-2)$$
and
$$C_4=2n+1+((n+1)r-1)(2R(k)-r)/(r^2-2r).$$
By the choices of $i_{7},i_{8}$ in \eqref{j}, the right hand side of \eqref{TAE} goes to infinity as $n \rightarrow \infty.$ This is a contradiction. We conclude that 
$$
|S_\omega'|=k \leq m-1.
$$
Note that for every $(x,y)$ in $S_\omega',$ we have $\zeta<1,$ 
by \eqref{defofS'}. 
Since we have $r$ different choices for $\omega$, and if $D> 0$, Lemma \ref{positiveDforms} implies one or two possible choices for $\omega$, 
we get
$$
N_F(h) \leq
\begin{cases}
r(m-1)+r\ {\rm if}\ D<0\\
2(m-1)+2\ {\rm if}\ D>0, r\ {\rm is\ even\ and}\ F \ {\rm is\ indefinite}\\
m\ {\rm if}\ D>0, r\ {\rm is\ odd\ and}\ F \ {\rm is\ indefinite}\\
\end{cases}
$$
Thus our proof is complete.
\qed

\subsection*{Proof of Corollary \ref{newcor}}
Let $\alpha_{1}$ and $\alpha_{2}$ be as in Theorem \ref{thml}.
We have 
$$
\frac{1}{2(r-1) + \alpha_{2}} = \frac{1}{2(r-1)} - \frac{\alpha_{2}}{2(r-1) [2(r-1) + \alpha_{2}]}.
$$
Since $\alpha_{2} > 0$, and by our assumption on the size of $\epsilon$ in the statement of Corollary \ref{newcor},  we have
$$  \frac{\alpha_{2}}{2(r-1) [2(r-1) + \alpha_{2}]} <  \frac{\alpha_{2}}{4(r-1)^2 }< \epsilon.
$$
Therefore, since $\alpha_{2} > 0$ and $\frac{r+\alpha_{1}}{2(r-1)+\alpha_{2}} < 7$, we obtain 
$$
0 < h \leq \frac{|\Delta|^{\frac{1}{2(r-1)} - \epsilon}}{2^{\frac{r}{2}}\, r^7} <  \frac{|\Delta|^\frac{1}{2(r-1) + \alpha_{2}}}{ 2^{\frac{r^2-r}{2(r-1) + \alpha_{2}}}\, r^{\frac{r+\alpha_{1}}{2(r-1)+\alpha_{2}} }  }.
$$
So by choosing  $\epsilon$ so that $$
 \frac{\alpha_{2}}{4(r-1)^2 } =\frac{(r^2+r+2)}{4(r-1)[(r-1)^{m-1}-2r-1]}<\epsilon<\frac{1}{2(r-1)},
$$ 
the integer $h$ will  satisfy \eqref{Delta_condn}  (see \eqref{Deltaprime} for definition of $\Delta'$) 
 and one can apply Theorem \ref{thml}.

Now assume that $D < 0$ and for a given value of $\epsilon > 0$,
$$0 < h \leq \frac{|\Delta|^{\frac{1}{2(r-1)} - \epsilon}}{2^{\frac{r}{2}}\, r^7}.$$
We are looking for $m \geq 3$ so that 
$$
\frac{|\Delta|^{\frac{1}{2(r-1)} - \epsilon}}{2^{\frac{r}{2}}\, r^7} <  \frac{|\Delta|^\frac{1}{2(r-1) + \alpha_{2}}}{ 2^{\frac{r^2-r}{2(r-1) + \alpha_{2}}}\, r^{\frac{r+\alpha_{1}}{2(r-1)+\alpha_{2}} }  },
$$
that is  $\frac{\alpha_{2}}{4(r-1)^2 } =\frac{(r^2+r+2)}{4(r-1)[(r-1)^{m-1}-2r-1]}<\epsilon$. Since
$$(r-1)^{m-2} <  (r-1)^{m-1} -2r -1$$
and
$$(r-1)^3 > r^2+r+2,$$
we may choose
$$
\epsilon >  \frac{(r-1)^{3}} {4 (r-1) (r-1)^{m-2}}= \frac{1}{4 (r-1)^{m-4}}.
$$
This means, the integer $m$ must be chosen to satisfy
$$
(m-4) >  \frac{\log \frac{1}{\epsilon} - \log 4}{\log(r-1)}.
$$

\qed

\section{Proof of Theorem \ref{2r1}}\label{proofof2r1}
Let  $F(x , y)$  be a diagonalizable form given in \eqref{form}.
As in Definition \ref{defofx,y0}, let $(x_0,y_0)$ be the solution with the largest $\zeta$ value $\zeta_0$. As before, let $S$ be the set of all solutions to the inequality $0 < |F(x , y)| \leq h$ and  $\omega$  a fixed $r$-th root of unity. We define  $S''_{\omega}$ to  be the set of solutions in $S\setminus\{(x_0,y_0)\}$ that are related to $\omega$.  In order to prove Theorem  \ref{2r1}, we will show that $|S''_{\omega}| \leq 2$.  The proof is similar to the proof of Theorem \ref{thml}. 

By Lemma \ref{Except1}, since $|j|>2h^{2/r},$ all solutions in 
$S_\omega''$ have $\zeta<1.$ Also
         \be\label{12r1}
	Z_i\geq\frac{|j|^{1/2}}{2^{1/2}h^{1/r}} \textrm{ for } i\geq 1.
	\ee
Suppose $|S_\omega''|\geq 3.$ Let $(x_1,y_1),(x_2,y_2),(x_3,y_3)\in S_\omega''$, with $\zeta_{1}\geq \zeta_{2} \geq \zeta_{3}$. Note that the proof of
Lemma \ref{gap_Siegel} is valid for the elements of $S_\omega''$ as well. Hence
	\be\label{z3z2}
	Z_3\geq \frac{|j|}{2h}Z_2^{r-1} \textrm{ and } Z_2\geq \frac{|j|}{2h}Z_1^{r-1}.
	\ee
Let $\Sigma_{n,g}$ be as in Section \ref{AN}, with $(x_1,y_1)$ and $(x_2,y_2)$ replaced by $(x_2,y_2)$ and $(x_3,y_3)$, respectively. Let
	$$\Lambda'_{n,g}=\Xi_{n,g}+\Pi_{n,g}$$
where
	$$\Xi_{n,g}=c_{1}(n,g)  \,h Z_{2}^{nr+1-g}Z_{3}^{-r+1}$$
and
	$$\Pi_{n,g}=c_{2}(n,g) \, 
	h^ {2n+1-g}  Z_{2}^{-r(n+1-g)+1-g}Z_{3}.
	$$
By \eqref{delta2r1} and \eqref{deltainrj}, we get
	\be\label{22r1}
	|j|\geq 2r^{i_{7}/r}h^{i_{8}/r}
	\ee
with
	\be\label{32r1}
	i_{7}=\frac{13r^2}{r^2-5r-2} \textrm{ and } i_{8}=\frac{2(3r-1)(r-2)}{r^2-5r-2}.
	\ee
Taking $t=3$ and $(j_1,j_2,j_3)=(1,2,3)$ in Lemma \ref{43}, \eqref{22r1} and \eqref{32r1} imply that $\zeta_2<1/2$. Similarly to  Section \ref{AAL},  by applying  Lemma \ref{c} with $(Z_1,Z_2)$ replaced by $(Z_2,Z_3)$,
we obtain
	\[
	Z_3\geq \frac{Z_{2}^{r(n+1-g)-1+g}}{2^{n+4}r^{(r(2g+3n)+2)/(r-2)}|j|^{(r(g+n)+2)/(r-2)}
	h^{2n+1-g}}
	\]
if
	\be\label{52r1}
	Z_3^{r-1}\geq 2^{3n+4}r^{(r(2g+3n)+2)/(r-2)}|j|^{(r(g+n)+2)/(r-2)}
	h Z_2^{nr+1-g}.
	\ee
As in Lemma 8.1, for a given integer $n \geq 1$ and $g \in \{0,1\},$ let 
$a_i=a_i(n,r,g),1 \leq i \leq 5,$ be some rational numbers which will be chosen.
	 We say that  $P[a_1,a_2,a_3,a_4,a_5]$ holds if
 \[
	Z_3\geq\frac{Z_2^{a_1}}{2^{a_2}r^{a_3}|j|^{a_4}h^{a_5}}.
	\]
Let  $A_1:=a_1(r-1)-nr-1+g>0$ and assume that $P[a_1,a_2,a_3,a_4,a_5]$ holds with $a_2+a_4\geq 0$. Then the inequality  \eqref{52r1} holds if
	\be\label{62r1}
	Z_1^{A_1(r-1)}\ |j|^{A_1-a_4(r-1)-
\frac{r(g+n)+2}{r-2}}
\geq
	\ee
	$$
	2^{A_1+	    a_2(r-1)+3n+4}r^{a_3(r-1)+\frac{r(2g+3n)+2}{r-2}}
	h^{A_1+a_5(r-1)+1}.
	$$
Using \eqref{12r1}, we see that \eqref{62r1} is valid if
	\[
	|j|^{B_1}\geq 2^{B_2}r^{B_3}h^{B_4}
	\]
where
	$$B_1=
	\frac{(r+1)}{2}A_1-a_4(r-1)-\frac{r(g+n)+2}{(r-2)},$$
	$$B_2= 
	\frac{(r+1)}{2}A_1+a_2(r-1)+3n+4,$$
	$$B_3=a_3(r-1)+\frac{r(2g+3n)+2}{r-2},
	$$
and 
	$$B_4=(2-1/r)A_1+a_5(r-1)+1.$$
By \eqref{22r1} and \eqref{32r1},
we have
	\be\label{83}
	|j|\geq 2^{(r+3)(r-2)/(r^2-5r-2)}r^{10r/(r^2-5r-2)}h^{2(3r-1)(r-2)/(r(r^2-5r-2))}.
	\ee
We implement the induction procedure given in Section \ref{AAL} with the above values of $B_1,\cdots ,B_4$ and $a_1,\cdots ,a_5$ as already given in the procedure. The conditions $(i)$--$(iv)$ are satisfied at every stage of the induction. Further, by \eqref{uk1} with $k=2$ and using \eqref{12r1} and \eqref{83}, it follows that \eqref{conc} is true.
Hence we get that the property 
$P[(n+1)r-1,n+4,(3nr+2)/(r-2),(nr+2)/(r-2),2n+1]$ is valid for any 
$n\geq 1$. Thus we have
	\begin{equation}\label{theaboveinequality76}
	Z_3\geq \frac{Z_2^{(n+1)r-1}}{2^{n+4}r^{(3nr+2)/(r-2)}|j|^{(nr+2)/(r-2)}h^{2n+1}}.
	\end{equation}
By \eqref{z3z2}, \eqref{22r1} and \eqref{32r1},   the
right hand side of the inequality \eqref{theaboveinequality76} tends to infinity as $n$
approaches infinity. This is a contradiction. Therefore, we conclude that for every given $\omega$,
$$|S''_\omega|\leq 2.$$

Counting the solution $(x_{0}, y_{0})$ (See Definition \ref{defofx,y0}), by Lemma \ref{positiveDforms} and the fact that there are generally $r$ choices for 
$\omega$, our proof is complete.\qed

\section{On Large solutions; proofs of Theorems \ref{thmy} and \ref{thmH}}\label{onlargesol}
\subsection*{Proof of Theorem \ref{thmy}}

The proof is similar to the proof  of Theorem \ref{thml}.  We will use 
 the proof of Lemma \ref{induction}. Let $\omega$ be a fixed $r$-th root of unity. Suppose that there are $k$ solutions $(x_{i}, y_{i})$, $i=1, \ldots, k$, that are related to $\omega$
with  $y_i\geq Y_L$ for all $i$. Also assume that the solutions are indexed such that $\zeta_{i+1} \leq \zeta_{i}$.  Let $m$ be the integer in the statement of Theorem  \ref{thmy}.  We will show that $k \leq m$. Let us assume $$m< k.$$ Then \eqref{ycondn} implies that 

\begin{equation}\label{ycondn2}
	y_{i} > \frac{r^{i_1} h^{i'_2}}{|j|^{i_3}}\ \textrm{ for }\ i\geq 1,
        \end{equation}
where
$$i_1=2+\frac{2}{r},\
i'_2= \frac{1}{r-2}+\frac{r-3}{(r-2)(r-1)^{k-2}}$$
and
\[
i_3=\begin{cases}
0 &\mbox{if } |j| \geq 1\\
\frac{r}{2(r-2)} & \textrm{otherwise}.
\end{cases}
\]
By Lemma \ref{negativedisc} we have
\be\label{u1y}
Z_i\geq \frac{|j|^{1/2}y_i}{2} \textrm{ for } 1\leq i\leq k.
\ee
This, together with \eqref{ycondn2}, implies that $Z_i^{r}>2h$. Hence $\zeta_i<1/2$ for all $i$.
 The inequality  \eqref{u1} holds if
\begin{equation}\label{y2rhj}
|j|^{B_6}y_1^{A_1R(k-1)}\geq 2^{B_2}r^{B_3}h^{B_5}
\end{equation}
where $A_1,B_2,B_3$ are as in Lemma \ref{induction};
$$B_5=A_1\frac{R(k-1)-1}{r-2}+a_5(r-1)+1$$
and
$$B_6=A_1\left(\frac{rR(k-1)-2}{2(r-2)}\right)-a_4(r-1)-\frac{r(g+n)+2}{r-2}.$$
Thus \eqref{y2rhj} is true provided the conditions\\
$ (v)\ A_1 R(k-1)\times i_1 \geq B_3+B_2/2$\\
$(vi) \ A_1 R(k-1)\times i_2' \geq B_5$\\
$(vii)\ 0\leq B_6\leq A_1 R(k-1)\times \frac{r}{2(r-2)}$\\
hold. 
We implement the induction procedure given in Section \ref{AAL} with the above values $A_1$, $B_2$, $B_3$, $B_5$, $B_6$ and $a_1,\cdots ,a_5$ as already given in the procedure. The conditions $(v)$--$(vii)$ are satisfied at every stage of the induction. Hence \eqref{Xi} is valid.
Further by the assumptions on $k$ and $r$, \eqref{conc} is valid by
\eqref{uk1} with $k$ replaced by $k-1$, \eqref{u1y} and \eqref{ycondn2}, thereby completing the induction.
Similarly from \eqref{ind_star}, we conclude that  for every $n \geq 1,$
\begin{eqnarray*}
&&Z_k \geq \\
&&|j|^{\left(\frac{rR(k-1)-2}{2(r-2)}-i_3R(k-1)\right)((n+1)r-1)-
\frac{nr+2}{r-2}}\times\\
&&r^{\frac{(3r^2-3r-8)R(k-1)+r}{2r(r-2)}((n+1)r-1)-\frac{7nr+4r-2n-4}{2(r-2)}}
h^{(n+1)r-2n-2}.
\end{eqnarray*}
The right hand side of the above
inequality goes to infinity as $n \rightarrow \infty.$ This is a contradiction. Therefore we have
$$
k\leq m.
$$
The rest of the proof is as in the proof of Theorem \ref{thml}.\qed

\subsection*{Proof of Theorem \ref{thmH}}

The proof is similar to the proof  of Theorem \ref{thmy}.  We will use 
 the proof of Lemma \ref{induction}. Let $\omega$ be a fixed $r$-th root of unity. Suppose that there are $k$ solutions $(x_{i}, y_{i})$, $i=1, \ldots, k$, that are related to $\omega$
with 
$|H(x_i,y_i)|\geq H_L$ for all $i$.  Assume that $m< k$. 
By \eqref{Hinrj}, we get
\be\label{67H}
Z_i\geq\left(\frac{|H_i|}{r^2(r-1)^2|j|^2}\right)^{1/(2r-4)} \textrm{ for } 1\leq i\leq k.
\ee
This, together with \eqref{HL},$r\geq 5, m\geq 3$ and our assumption  $m< k$, implies that $Z_i^{r}>2h$. Hence $\zeta_i<1/2$ for all $i$.
 The inequality \eqref{u1} holds if
\begin{equation}\label{HLB7}
|j|^{B_7}r^{B_8}h^{B_9}\geq 1,
\end{equation}
where
$$B_7=A_1\frac{i_6R(k-1)-2}{2(r-2)}-a_4(r-1)-\frac{r(g+n)+2}{r-2},$$
$$B_8=(i_4-4)\frac{A_1R(k-1)}{2(r-2)}-B_3-\frac{B_2'}{2},$$
$$B_9=i_5\frac{A_1R(k-1)}{2(r-2)}-B_4',$$
$A_1$, $B_3$ are as before,
	\[
	B_2'=A_1\frac{R(k-1)-1}{r-2}+a_2(r-1)+3n+4
	\]
and
	\[
	B_4'=A_1\frac{R(k-1)-1}{r-2}+a_5(r-1)+1.
	\]
	The function $R(.)$ is as in Definition \ref{defofR(k)}.
Thus \eqref{HLB7} is valid provided the conditions
$ (viii)\ B_3+B_2'/2 \leq (i_4-4)\frac{A_1R(k-1)}{2(r-2)}$\\
$(ix) \ B_4' \leq i_5\frac{A_1R(k-1)}{2(r-2)}$\\
$(x)\ A_1\frac{R(k-1)-1}{r-2}-a_4(r-1)-\frac{r(g+n)+2}{r-2}\geq 0$\\
hold. 
We implement the induction procedure as in the proof of Theorem \ref{thml} and find that the conditions $(viii)$--$(x)$ are satisfied at every stage of the induction. Hence \eqref{Xi} is valid.
Also, since $k>m \geq 3,$ \eqref{conc} is valid by \eqref{uk1} with $k$ replaced by $k-1$,
\eqref{67H} and \eqref{HL}.
This completes the induction.
Similarly from \eqref{ind_star} we get for any $n\geq 1,$  
\begin{eqnarray*}\label{lastuk}
&&Z_k \geq\\
&&|j|^{\left(\frac{i_6R(k-1)-2}{2(r-2)}\right)((n+1)r-1)-\frac{nr+2}{r-2}} 
r^{\left(\frac{11r-2}{2(r-2)}\right)((n+1)r-1)-\frac{7nr+4r-2n-4}{2(r-2)}}
h^{(n+1)r-2n-2}.\end{eqnarray*}
The right hand side of the above
inequality goes to infinity as $n \rightarrow \infty.$ This contradiction implies that 
$$
k\leq m.
$$
The rest of the proof is as in the proof of Theorem \ref{thml}.\qed


\section{Proof of Theorem \ref{Binoineq}}\label{proofofBinoineq}

Let $a$ and $b$ be positive integers. 
The discriminant $\Delta$ of the form $F(x , y) = ax^r - by^r$ is equal to
$$
(-1)^{(r-1)(r+2)/2} r^r \left( ab \right)^{r-1}.
$$
 For even degree $r$, notice that if $(x , y)$ is a solution to the inequality $0< |ax^r - by^r| \leq c$, so is $(-x , y)$. Also if $(x , y)$ is related to $\omega$ then $(-x , y)$ is related to $-\omega$.  In Theorem \ref{Binoineq}, we are interested in positive solutions $x$ and $y$, therefore we only need to count the number of solutions related to one $r$-th root of unity in this case (see Lemma \ref{positiveDforms} and its proof).
 
 The above observations and Corollary \ref{cork=3} imply Theorem \ref{Binoineq}.
\qed
   \section{Diagonalizable Thue equations, proof of Theorem \ref{BSchfordiag}}\label{DTE}

\begin{prop}\label{BSCH}
Let $\mathfrak{S}_{r}$ be the set of  diagonalizable binary forms $F(x , y) \in \mathbb{Z}[x , y]$ of degree $r \geq 3$.
Let $\mathfrak{N}$ be an upper bound for the number of solutions of Thue equations
$$
|F(x , y)| = 1
$$
as $F$ varies over the elements of $\mathfrak{S}_{r}$. Then for $h \in \mathbb{N}$ and $G(x , y) \in \mathfrak{S}_{r}$, the equation 
$$|G(x , y)| = h$$
has at most $$\mathfrak{N} \, r^{\omega(h)}$$ primitive solutions, where $\omega(h)$ is the number of prime divisors of $h$. 
\end{prop}
\begin{proof}
This is essentially a special case of Bombieri and Schmidt's result in \cite{Bos}, where they showed that if $N_{n}$ is an upper bound for the number of solutions to the equations $|F(x, y)| = 1$, as $F(x , y)$ varies over irreducible binary forms of degree  $n$ with integer coefficients  then $N_{n}n^{\omega(h)}$ is an upper bound  for the number of primitive solutions to $|F(x, y)| = h$. Bombieri and Schmidt proved this fact by reducing a given Thue equation $|F(x , y)| = h$ modulo every prime factor of the integer $h$.  This reduction is explained in the proof of Lemma 7 of \cite{Bos}, where the form $F(x , y)$ of degree $n$ is reduced to some other binary forms  of degree $n$.  These reduced forms  are obtained through the action of $2 \times 2$   matrices with rational entries and non-zero discriminant  on the binary form $F(x , y)$.  We refer the reader to \cite{Bos} and \cite{Ste} for more details. It is clear (see Section \ref{EF}) that under the action of $2 \times 2$ matrices a diagonalizable form will be reduced to other diagonalizable forms. 
\end{proof}

The following lemma says that if, following the reduction method  that was mentioned  in the proof of Proposition \ref{BSCH}, we reduce an equation $|F(x , y)| = h$  to a family of equations $|\tilde{F}(x , y)|= 1$, then the absolute values of discriminants of the forms $\tilde{F}(x, y)$ will be bounded from below by a function of $h$ and the discriminant of $F$.  This is helpful in our applications, as we will use our main theorems, such as Theorem \ref{2r1}, with assumption on the size of the discriminant of $F$.

\begin{lemma}\label{howDeltachanges}
Let $F(x , y) \in \mathbb{Z}$, with degree $r$ and  discriminant $\Delta(F)$. Assume that  $h\in \mathbb{Z}$, 
with $\gcd(h, \Delta(F))=1$. Then there is a set $W$ of binary forms of degree $r$, with $|W|\leq r^{\omega(h)}$, so that each primitive solution $(x , y)$ to the equation
$$|F(x , y)| = h$$
corresponds to a unique triple $(\tilde{F}, x', y')$ where $\tilde{F} \in W$, 
$$
|\tilde{F}(x', y')| = 1
$$
and
$$
|\Delta(\tilde{F})| \geq  h^{(r-1)(r-2)} |\Delta(F)|.
$$
\end{lemma}
\begin{proof}
See \cite{Ste} for proof. In particular, page 810, as well as Theorem 1 and the definitions presented immediately  after that on page 795 of \cite{Ste}, where it is shown that 
$$
|\Delta(\tilde{F})| \geq  \frac{ h^{(r-2)(r-1)}|\Delta(F)|}{G(h, r, \Delta)^{r (r-1) } },
$$
where the function $G$ is defined on page 795 of \cite{Ste}. Moreover, on page 795 of \cite{Ste}, one can see that
$$
 G(h, r, \Delta) =1,
$$
if $\gcd(\Delta , h)$ = 1.
\end{proof}

\subsection{Proof of Theorem \ref{binomial=}}
Let $a$ and $b$ be positive integers. 
The discriminant $\Delta$ of the form $F(x , y) = ax^r - by^r$ is equal to
$$
(-1)^{(r-1)(r+2)/2} r^r \left( ab \right)^{r-1}.
$$
Note that $D(F) > 0$ and by assumption $\gcd(c , r a b) =1$, we have $\gcd(\Delta, c) = 1$.
By Lemma \ref{howDeltachanges}, we may  apply  Theorem \ref{siegels} to equations $|\tilde{F}(x , y)| =1$, which are obtained from reducing the equation $|F(x , y)| = |ax^r - by^r| = c$, to conclude that each of these new diagonalizable Thue equations have at most $2$ or $4$ solutions, if $r$ is odd or even, respectively. It is important to note here that once we reduce an equation  $|F(x , y)| = c$ with $D(F) > 0$,  to equations $|\tilde{F}(x , y) | =1$, we have $D(\tilde{F}) > 0$, as these reductions are the result of a number of actions of $2 \times 2$  matrices. By definition of $D$ in \eqref{defofD}, it is clear that $D(\tilde{F})$ is the product of $D(F)$ and the squares of the determinants of these real matrices.

 By Proposition \ref{BSCH}, the equation $|ax^r - by^r| = c$ has at most $2 r^{\omega(c)}$ primitive solutions if $r$ is odd,  and at most $4 r^{\omega(c)}$ primitive solutions  if $r$ is even.
 For even degree $r$, notice that  $(x , y)$ is a solution to the inequality $ |ax^r - by^r| = c$ if and only if  $(-x , y)$.   In Theorem \ref{binomial=}, we are interested in positive solutions $x$ and $y$, therefore if $r$ is even, we can divide the total number of solutions by $2$.
 
 Now assume that 
 $$
 ab \geq 2^{r} r^{7r/(r-4)}.
 $$
Instead of Siegel's Theorem \ref{siegels}, we will use our Corollary \ref{cork=3}. Similarly, we conclude that  $|ax^r - by^r| = c$ has $3 \, r^{\omega(c)}$ primitive solutions  if $r$ is odd and at most $6\,  r^{\omega(c)}$ primitive solutions if $r$ is even, and therefore the number of positive solutions $x$ and $y$ is bounded by $3 r^{\omega(c)}$.

\qed

\subsection{Proof of Theorem \ref{BSchfordiag}}
By Lemma \ref{howDeltachanges}, every primitive solution $(x , y)$ to the equation
$$|F(x , y)| = h$$
corresponds to a unique triple $(\tilde{F}, x', y')$ where $\tilde{F} \in W$, 
$$
|\tilde{F}(x', y')| = 1
$$
and
$$
|\Delta(\tilde{F})| \geq h^{(r-1)(r-2) } |\Delta(F)|,
$$
where the binary forms  $\tilde{F}$ are diagonalizable, by Proposition \ref{BSCH}. Using the assumption 
\[
		|\Delta(F)| \geq 2^{r^2-r} r^{r+7r(r-1)/(r-4)},
	\]
in the statement of Theorem \ref{BSchfordiag},
we conclude that
$$
|\Delta(\tilde{F})| >  2^{r^2-r} r^{r+7r(r-1)/(r-4)}.
$$
 By Corollary \ref{cork=3}, and taking $h=1$, we obtain an upper bound for the number of solutions to each Thue equation 
 $$
|\tilde{F}(x', y')| = 1.
$$ 
Since we have  at most $r^{\omega(h)}$ equations $
|\tilde{F}(x', y')| = 1$, we obtain the desired result.
\qed

\section*{Acknowledgements}

The authors are indebted to the anonymous referee for reading this manuscript carefully and providing several insightful comments, which improved the content and the  presentation. In particular, the referee's remarks improved our  Theorems 1.7 and 1.8, as well as Lemma 7.2 and  some subsequent lemmas.

The present work  was started when N. Saradha visited the University of Oregon from June 15 to June 26, 2015. She would like to thank Shabnam Akhtari for the invitation. The computations in the paper were done by N. Saradha and Divyum Sharma at TIFR, Mumbai. They thank their home institution for providing computing facilities.

Shabnam Akhtari's research is partly supported by the National Science Foundation grant DMS-1601837.


\end{document}